\documentclass[a4paper,12pt]{article}

\usepackage[english]{babel}
\usepackage[T1]{fontenc}

\usepackage{tikz}
\tikzstyle{vertex}=[circle, draw, inner sep=2pt, minimum size=6pt]

\usepackage{mathrsfs}

\usepackage[a4paper,top=3cm,bottom=2cm,left=3cm,right=3cm,marginparwidth=1.75cm]{geometry}

\usepackage{comment}
\usepackage{amsfonts}
\usepackage{fullpage}
\usepackage{latexsym}
\usepackage{amsfonts}
\usepackage{blkarray}
\usepackage{graphicx}
\usepackage{float}
\usepackage{enumerate}
\usepackage{amsthm,amsmath,amssymb}
\usepackage{verbatim}
\usepackage{enumitem}
\usepackage{blkarray}
\usepackage{mathtools}
\usepackage{algorithm}
\usepackage[noend]{algpseudocode}

\providecommand{\keywords}[1]{
  \small	
  \textbf{\textit{Keywords---}} #1
}

\newtheorem{theorem}{Theorem}[section]
\newtheorem{lemma}{Lemma}[section]
\newtheorem{corollary}{Corollary}[section]
\newtheorem{proposition}{Proposition}[section]

\newtheorem{example}{Example}[section]

\newcommand{\nc}{\newcommand}
\newcommand{\overbar}[1]{\mkern 1.5mu\overline{\mkern-1.5mu#1\mkern-1.5mu}\mkern 1.5mu}
\nc{\cA}{{\cal A}}
\nc{\cB}{{\cal B}}
\nc{\cC}{{\cal C}}
\nc{\cD}{{\cal D}}
\nc{\cE}{{\cal E}}
\nc{\cG}{{\cal G}}
\nc{\cF}{{\cal F}}
\nc{\cH}{{\cal H}}
\nc{\cI}{{\cal I}}
\nc{\cK}{{\cal K}}
\nc{\cL}{{\cal L}}
\nc{\cM}{{\cal M}}
\nc{\cN}{{\cal N}}
\nc{\cO}{{\cal O}}
\nc{\cP}{{\cal P}}
\nc{\cQ}{{\cal Q}}
\nc{\cR}{{\cal R}}
\nc{\cS}{{\cal S}}
\nc{\cT}{{\cal T}}
\nc{\tx}{{\tilde x}}
\nc{\la}{{\langle}}
\nc{\ra}{{\rangle}}
\nc{\ts}{\textsuperscript}
\def\R{\mathbb{R}}

\nc{\bea}{\begin{eqnarray}}
\nc{\eea}{\end{eqnarray}}
\nc{\bean}{\begin{eqnarray*}}
\nc{\eean}{\end{eqnarray*}}
\nc{\be}{\begin{equation}}
\nc{\ee}{\end{equation}}
\nc{\ben}{\begin{equation*}}
\nc{\een}{\end{equation*}}
\nc{\ba}{\begin{array}}
\nc{\ea}{\end{array}}
\title{On Standard Quadratic Programs with Exact and Inexact Doubly Nonnegative Relaxations}
\author{Y. G\"{o}rkem G\"{o}kmen\thanks{Department of Industrial Engineering, Izmir University of Economics, 35330 Bal\c{c}ova, Izmir, Turkey. E-mail: \tt{gorkemgokmen@gmail.com}} \and E. Alper Y{\i}ld{\i}r{\i}m\thanks{School of Mathematics, Peter Guthrie Tait Road, The University of Edinburgh, Edinburgh, EH9 3FD, United Kingdom (Corresponding author). ORCID ID: 0000-0003-4141-3189 E-mail: \tt{E.A.Yildirim@ed.ac.uk}}}
\date{\today}
\begin{document}

\maketitle

\begin{abstract}
The problem of minimizing a (nonconvex) quadratic form over the unit simplex, referred to as a standard quadratic program, admits an exact convex conic formulation over the computationally intractable cone of completely positive matrices. Replacing the intractable cone in this formulation by the larger but tractable cone of doubly nonnegative matrices, i.e., the cone of positive semidefinite and componentwise nonnegative matrices, one obtains the so-called doubly nonnegative relaxation, whose optimal value yields a lower bound on that of the original problem. We present a full algebraic characterization of the set of instances of standard quadratic programs that admit an exact doubly nonnegative relaxation. This characterization yields an algorithmic recipe for constructing such an instance. In addition, we explicitly identify three families of instances for which the doubly nonnegative relaxation is exact. We establish several relations between the so-called convexity graph of an instance and the tightness of the doubly nonnegative relaxation. We also provide an algebraic characterization of the set of instances for which the doubly nonnegative relaxation has a positive gap and show how to construct such an instance using this characterization.     
\end{abstract}



\section{Introduction} \label{intro}

A standard quadratic program, which involves minimizing a
(nonconvex) quadratic form (i.e., a homogeneous quadratic
function) over the unit simplex, can be expressed as
\[
 \textrm{(StQP)} \quad \nu(Q) = \min\left\{x^T Q x: x \in \Delta_n\right\},
\]
where $Q \in \cS^n$ and $\cS^n$ denotes the space of $n \times n$ real symmetric matrices, and $ \Delta_n$ denotes the unit simplex in the $n$-dimensional Euclidean space $\R^n$, i.e., 
\begin{equation} \label{def_delta}
\Delta_n = \{x \in \R^n: e^T x = 1,~x \ge 0\},
\end{equation}
where $e \in \R^n$ is the vector of all ones. 

The standard quadratic program was singled out by Bomze \cite{ref:Bomze1998}, who also described several properties of the problem. It has many application areas such as portfolio optimization \cite{ref:markowitz1952}, population genetics \cite{ref:kingman1961}, evolutionary game theory \cite{ref:Bomze2002}, and maximum (weighted) clique problem \cite{motzkin1965maxima,ref:Gibbons1997}. Since (StQP) contains the maximum (weighted) clique problem as a special case, the problem is, in general, NP-hard.

A standard quadratic program admits an exact reformulation as a linear optimization problem over the convex cone of completely positive matrices~\cite{ref:Bomze2000} (see Section \ref{CP_and_DN_form}). Since the cone of completely positive matrices is computationally intractable~\cite{ref:DickGijb}, replacing this conic constraint by a larger but computationally tractable convex cone immediately gives rise to a relaxation, whose optimal value yields a lower bound on that of (StQP). 

In this paper, we focus on the so-called \emph{doubly nonnegative relaxation} of (StQP), which arises from replacing the cone of completely positive matrices in the aforementioned reformulation by the larger cone of doubly nonnegative matrices, i.e., the cone of positive semidefinite and componentwise nonnegative matrices. In contrast with the cone of completely positive matrices, a linear optimization problem over the cone of doubly nonnegative matrices can be solved to within an arbitrary accuracy in polynomial time. For a given optimization problem, a relaxation is said to be exact if the lower bound arising from that relaxation agrees with the optimal value of the original problem. Our main objective is to provide a characterization of the set of instances of (StQP) that admit an exact doubly nonnegative relaxation as well as a characterization of the set of instances for which the relaxation has a positive gap. Note that such characterizations shed light on instances of (StQP) that can be solved in polynomial time. Furthermore, they are helpful for identifying supporting hyperplanes of the feasible region of the convex conic reformulation of (StQP) that are common with those of the feasible region of the doubly nonnegative relaxation. 

Our contributions in this paper are as follows. 
\begin{enumerate}
    \item We present a full characterization of the set of instances of (StQP) that admit an exact doubly nonnegative relaxation (see Section~\ref{exact_dnn}).
    \item Based on this characterization, we propose a simple algorithmic recipe for generating an instance with an exact doubly nonnegative relaxation (see Section~\ref{exact_dnn}).
    \item We explicitly identify three families of instances of (StQP) with exact doubly nonnegative relaxations (see Section~\ref{subsets_of_cQ}). 
    \item We establish several relations between the maximal cliques of the so-called convexity graph of an instance and the tightness of the corresponding doubly nonnegative relaxation (see Section~\ref{max_clique_conv_graph}).
    \item We present an algebraic characterization of the set of instances of (StQP) for which the doubly nonnegative relaxation has a positive gap (see Section~\ref{pos_gap}).
    \item By using this characterization, we propose a procedure for generating an instance of (StQP) with a positive relaxation gap (see Section~\ref{pos_gap}). 
\end{enumerate}

This paper is organized as follows. We briefly review the related literature in Section~\ref{lit_rev} and define our notation in Section~\ref{notation}. In Section~\ref{prelim}, we review several known results and present the convex conic reformulation as well as the doubly nonnegative relaxation. Section~\ref{exact_dnn} is devoted to the characterization of instances of (StQP) with an exact doubly nonnegative relaxation. Using this characterization, we also a describe a procedure for generating an instance with an exact relaxation. In Section~\ref{subsets_of_cQ}, we identify three families of instances of (StQP) that admit an exact relaxation by relying on the characterization in Section~\ref{exact_dnn}. We define the convexity graph and establish several relations between the maximal cliques of this graph and the exactness of the doubly nonnegative relaxation in Section~\ref{max_clique_conv_graph}. In particular, we identify a sufficient condition that can be used to find an instance of (StQP) with an exact relaxation that is not covered by any of the three families in Section~\ref{subsets_of_cQ}. Section~\ref{pos_gap} presents an algebraic characterization of the instances of (StQP) with a positive relaxation gap and a procedure for generating such an instance. Finally, we conclude the paper in Section~\ref{conc}.

\subsection{Literature Review} \label{lit_rev}

We briefly review the related literature. A standard quadratic program can be equivalently formulated as a linear optimization problem over the cone of completely positive matrices, i.e., a copositive program~\cite{ref:Bomze2000}. Despite the fact that solving this conic reformulation remains NP-hard, it offers a fresh perspective for developing tractable approximations of (StQP) by instead focusing on tractable approximations of the cone of completely positive matrices. Relying on sum-of-squares decomposition, Parrilo~\cite{parrilo2000structured} proposed an approximation hierarchy, i.e., a sequence of nested convex cones that provide increasingly better inner approximations of the dual cone of copositive matrices, which, by duality, yields a sequence of increasingly better outer approximations of the cone of completely positive matrices. Since each of these cones can be represented by linear matrix inequalities, a linear optimization over each cone can be cast as a semidefinite program and can therefore be solved in polynomial time. In fact, the dual of the first cone in this hierarchy is precisely the cone of doubly nonnegative matrices. By exploiting weaker conditions, de Klerk and Pasechnik~\cite{de2002approximation} proposed a sequence of polyhedral cones that yield increasingly better outer approximations of the cone of completely positive matrices. For other inner and outer approximations, we refer the reader to~\cite{pena2007computing,bundfuss2009adaptive,alper2012accuracy,lasserre2014new,gouveia2019inner}.

By combining the approximations of (StQP) arising from the polyhedral approximation hierarchy of~\cite{de2002approximation} with a simple search on a finite grid on the unit simplex, Bomze and de Klerk~\cite{bomze2002solving} established a polynomial-time approximation scheme for (StQP). In~\cite{alper2012accuracy}, the second author of this paper proposed an inner polyhedral approximation hierarchy for the cone of completely positive matrices and tightened the error bound of~\cite{bomze2002solving} used to establish the polynomial-time approximation scheme. The resulting error bound also translates directly into an error bound on the gap between the optimal value of (StQP) and that of the doubly nonnegative relaxation since the hierarchy of Parrilo~\cite{parrilo2000structured} is stronger than that of~\cite{de2002approximation}. Sa\u{g}ol and Y{\i}ld{\i}r{\i}m~\cite{saugol2015analysis} studied the behavior of inner and outer polyhedral approximation hierarchies of~\cite{de2002approximation} and~\cite{alper2012accuracy} on standard quadratic programs. They presented algebraic characterizations of instances of (StQP) with exact inner and/or outer approximations at each level of these hierarchies and established several properties of such instances. In this paper, we aim to establish similar characterizations and properties of the set of instances of (StQP) that admit exact doubly nonnegative relaxations as well as those with a positive relaxation gap. Therefore, 
our focus in this paper is similar to that of~\cite{saugol2015analysis}. 

Very recently, Kim, Kojima, and Toh~\cite{KimKT2000} studied the doubly nonnegative relaxations of general copositive programs. Under the assumption that the correlative and sparsity patterns of the data matrices form a block-clique graph, they established the exactness of the doubly nonnegative relaxations. In particular, their results imply that the doubly nonnegative relaxation of any convex quadratically constrained quadratic program is exact. We note that the correlative and sparsity patterns of the data matrices of the copositive formulation of (StQP) form a complete graph, which is, indeed, a block-clique graph. On the other hand, the exactness of the doubly nonnegative relaxation in~\cite{KimKT2000} is established under the additional assumption that the size of each clique is at most four, which is only satisfied for the doubly nonnegative relaxation of instances of (StQP) with $n \leq 4$. However, for such instances, it is already known that the doubly nonnegative relaxation is exact (see Section~\ref{prelim}). Therefore, our results in this paper are not implied by the results in~\cite{KimKT2000}.

\subsection{Notation} \label{notation}

We use $\R^n, \R^n_+, \R^n_{++}$, $\R^{m \times n}$, and $\cS^n$ to denote the $n$-dimensional Euclidean space, the nonnegative orthant, the positive orthant, the set of $m \times n$ real matrices, and the space of $n \times n$ real symmetric matrices, respectively. The unit simplex in $\R^n$, given by \eqref{def_delta}, is denoted by $\Delta_n$. We reserve $e$ and $e_j$ for the vector of all ones and the $j$th unit vector, respectively. The matrix of all ones is denoted by $E = ee^T$ and $I$ denotes the identity matrix. The dimension will always be clear from the context. We use 0 to denote the real number 0, the vector of all zeroes, as well as the matrix of all zeroes. We use calligraphic letters to denote the subsets of $\cS^n$. We use uppercase boldface Roman or uppercase Greek letters to denote the subsets of $\R^n$. We use uppercase letters both for matrices and index sets, and lower case letters to denote vectors, dimensions, and indices of vectors and matrices. Scalars will be denoted by lowercase Greek letters, with the exception of $\ell(Q)$ that denotes the lower bound arising from the doubly nonnegative relaxation. For an index set $A \subseteq \{1,\ldots,n\}$, we denote by $|A|$ the cardinality of $A$. For $x \in \R^n$, $Q \in \cS^n$, $A \subseteq \{1,\ldots,n\}$, and $B \subseteq \{1,\ldots,n\}$, we denote by $x_A \in \R^{|A|}$ the subvector of $x$ restricted to the indices in $A$ and by $Q_{AB}$ the submatrix of $Q$ whose rows and columns are indexed by $A$ and $B$, respectively. Therefore, $Q_{AA}$ denotes a principal submatrix of $Q$. We use the simplified notations $x_j$ and $Q_{ij}$ for singleton index sets. For $v \in \R^n$, $v^\perp$ denotes the orthogonal complement of $v$. For any $U \in \R^{m \times n}$ and $V \in \R^{m \times n}$, the trace inner product is denoted by 
\[
\langle U, V \rangle := \sum\limits_{i=1}^m \sum\limits_{j = 1}^n U_{ij} V_{ij}.
\]

For an instance of (StQP) with $Q \in \cS^n$, we denote by $\nu(Q)$ the optimal value, and the set of optimal solutions is denoted by
\begin{equation} \label{def_Omega}
\Omega(Q) = \{x \in \Delta_n: x^T Q x = \nu(Q)\}. 
\end{equation}
For a given $x \in \Delta_n$, we define the following index sets:
\begin{eqnarray}
A(x) & = & \left\{j \in \{1,\ldots,n\}: x_j > 0 \right\}, \label{def_P}\\
Z(x) & = & \left\{j \in \{1,\ldots,n\}: x_j = 0 \right\}. \label{def_Z}
\end{eqnarray}

\section{Preliminaries} \label{prelim}

In this section, we review several known results from the literature and present the copositive formulation of a standard quadratic program as well as the doubly nonnegative relaxation. 

\subsection{Convex Cones}

We define the following cones in {$\cS^n$}:
\begin{eqnarray}
\label{def_N}
\cN^n & = & \left\{M \in \cS^n: M_{ij} \geq 0, \quad i = 1,\ldots,n;~j = 1,\ldots,n\right\}, \\
\label{def_PSD}
{\cal PSD}^n & = & \left\{M \in \cS^n: u^T M u \geq 0, \quad \forall u \in \R^n \right\}, \\
\label{def_COP}
{\cal COP}^n & = & \left\{M \in \cS^n: u^T M u \geq 0, \quad \forall u \in \R^n_+ \right\}, \\
\label{def_CP}
{\cal CP}^n & = & \left\{M \in \cS^n: M = \sum\limits_{k=1}^r b^k (b^k)^T, \quad \textrm{for some}~b^k \in \R^n_+,~k = 1,\ldots,r\right\}, \\
\label{def_DN}
{\cal DN}^n & = & {\cal PSD}^n \cap \cN^n, \\
\label{def_SPN}
{\cal SPN}^n & = &  \left\{M \in \cS^n: M = M_1 + M_2, \quad \textrm{for some}~M_1 \in {\cal PSD}^n,~M_2 \in \cN^n\right\},
\end{eqnarray}
namely, $\cN^n$ is the cone of componentwise nonnegative matrices, ${\cal PSD}^n$ is the cone of positive semidefinite matrices, ${\cal COP}^n$ is the cone of copositive matrices, ${\cal CP}^n$ is the cone of completely positive matrices, ${\cal DN}^n$ is the cone of doubly nonnegative matrices, and ${\cal SPN}^n$ is the cone of SPN matrices. i.e., the cone of matrices that can be decomposed into the sum of a positive semidefinite and a componentwise nonnegative matrix. Each of these cones is closed, convex, full-dimensional, and pointed, and the following set of inclusion relations is satisfied:
\begin{equation} \label{inc_rels}
{\cal CP}^n \subseteq {\cal DN}^n \subseteq \left\{ \begin{matrix} \cN^n \\ {\cal PSD}^n \end{matrix} \right\} \subseteq {\cal SPN}^n \subseteq {\cal COP}^n.
\end{equation}

By~\cite{ref:Diananda},
\begin{equation} \label{diananda}
{\cal CP}^n = {\cal DN}^n, \quad \textrm{and} \quad {\cal SPN}^n = {\cal COP}^n \quad \textrm{if and only if} \quad n \leq 4.
\end{equation}
For $n \geq 5$, checking membership is NP-hard for both ${\cal CP}^n$~\cite{ref:DickGijb} and ${\cal COP}^n$~\cite{ref:MurtKab}. Each of the remaining four cones is tractable in the sense that they admit polynomial-time membership oracles.

The following lemma collects several results that will be useful throughout the paper.

\begin{lemma} \label{lem_gen_rels}
Let ${\cal K}^n \in \left\{ {\cal CP}^n, {\cal DN}^n, \cN^n, {\cal PSD}^n, {\cal SPN}^n, {\cal COP}^n\right\}$. Then, the following relations are satisfied:
\begin{enumerate}
\item[(i)] If $U \in {\cal K}^n$, then $U_{kk} \geq 0,~k = 1,\ldots,n$.
\item[(ii)] $U \in {\cal K}^n$ if and only if $J^T U J \in {\cal K}^n$, where $J \in \R^{n \times n}$ is a permutation matrix.
\item[(iii)] $U  \in {\cal K}^n$ if and only if $D U D  \in {\cal K}^n$, where $D \in {\cal S}^n$ is a diagonal matrix with positive diagonal entries.
\item[(iv)] If $U \in {\cal K}^n$, then every principal $r \times r$ submatrix of $U$ is in ${\cal K}^r$, $r = 1,\ldots,n$.
\item[(v)] If $U_1 \in {\cal K}^n$ and $U_2 \in {\cal K}^m$, then 
\begin{equation} \label{direct_sum}
U_1 \oplus U_2 = \begin{bmatrix} U_1 & 0 \\0 & U_2 \end{bmatrix} \in {\cal K}^{n+m} .
\end{equation}
In particular, $U_2 = 0$ can be chosen.
\end{enumerate}
\end{lemma}

\subsection{Copositive Formulation and Doubly Nonnegative Relaxation} \label{CP_and_DN_form}

(StQP) can be formulated as a copositive program~\cite{ref:Bomze2000}, i.e., a linear optimization problem over an affine subset of the convex cone of completely positive matrices:
\[
 \textrm{(CP)} \quad \nu(Q) = \min\{\la Q, X \ra: \la E, X \ra =
 1, \quad X \in {\cal CP}^n\},
\]
where $X \in \cS^n$.

By \eqref{inc_rels}, we can replace the intractable conic constraint $X \in {\cal CP}^n$ by $X \in {\cal DN}^n$ and obtain a relaxation of (CP), or, equivalently, a relaxation of (StQP):
\[
 \textrm{(DN-P)} \quad \ell(Q) = \min\left\{\langle Q, X \rangle: \langle E, X \rangle = 1, \quad X \in {\cal DN}^n \right\},
\]
(DN-P) is referred to as the {\em doubly nonnegative relaxation} of (StQP). The Lagrangian dual problem of (DN-P) is given by
\[
\textrm{(DN-D)} \quad \ell(Q) = \max\left\{\sigma: \sigma E + S = Q, \quad S \in {\cal SPN}^n \right\},
\]
where $\sigma \in \R$ and $S \in \cS^n$. It is well-known that both (DN-P) and (DN-D) satisfy the Slater's condition, which implies that strong duality is satisfied, and that optimal solutions are attained in both (DN-P) and (DN-D).

For all $Q \in \cS^n$, we have 
\begin{equation} \label{lb_dnn}
\ell(Q) \leq \nu(Q),
\end{equation}
since ${\cal CP}^n \subseteq {\cal DN}^n$. For $n \leq 4$, we have $\ell(Q) = \nu(Q)$ by \eqref{diananda}. For $n \geq 5$, we are interested in the characterization of instances of (StQP) for which $\ell(Q) = \nu(Q)$ as well as those with $\ell(Q) < \nu(Q)$.

The following lemma presents a simple shift invariance property that will be useful throughout the remainder of the paper.

\begin{lemma} \label{shift_invariance}
For any $Q \in \cS^n$ and any $\lambda \in \R$,
\begin{eqnarray}
\label{shift1}
\nu(Q + \lambda E) & = & \nu(Q) + \lambda, \\
\label{shift2}
\ell(Q + \lambda E) & = & \ell(Q) + \lambda.
\end{eqnarray}
Furthermore, $\Omega(Q) = \Omega(Q + \lambda E)$.
\end{lemma}
\begin{proof}
The relations \eqref{shift1} and \eqref{shift2} immediately follow from the formulations (CP) and (DN-P), respectively, since $\langle Q + \lambda E, X \rangle = \langle Q, X \rangle + \lambda \langle E, X \rangle = \langle Q, X \rangle + \lambda$ for any $X \in \cS^n$ such that $\langle E, X \rangle = 1$. The last assertion directly follows from the observation that
\[
 x^T(Q + \lambda E) x = x^TQ x + \lambda x^T E x =  x^T Q x + \lambda (e^T x)^2  = x^T Q x + \lambda
\]
for any $\lambda \in \R$ and $x \in \Delta_n$.
\end{proof}

By Lemma~\ref{shift_invariance}, if $\ell(Q) = \nu(Q)$ for a given $Q \in \cS^n$, note that $\ell(Q + \lambda E) = \nu(Q + \lambda E)$ for any $\lambda \in \R$. We will repeatedly use this observation in the remainder of the manuscript. 

\subsection{Local Optimality Conditions}

In this section, we review the local optimality conditions of (StQP). 

Given an instance of (StQP), $x \in \R^n$ is a local minimizer if and only if there exists $s \in \R^n$ such that the following conditions are satisfied (see, e.g.,~\cite{ref:majthay1971,ref:jiaquan1982}):
\begin{eqnarray} \label{kkt}
Q x - \left( x^T Q x \right) e - s & = & 0, \label{eq1}\\
e^T x & = & 1, \label{eq2} \\
x & \in & \R^n_+, \label{eq3} \\
s & \in & \R^n_+, \label{eq4} \\
x_j s_j & = & 0, \quad j = 1,\ldots,n, \label{eq5} \\
d^T Q d & \geq & 0, \quad \textrm{for all}~d \in \mathbf{D}(x), \label{eq6}
\end{eqnarray}
where 
\begin{equation} \label{def_D}
\mathbf{D}(x) = \left\{ d \in \R^n: e^T d = 0, \quad d^T Q x = 0, \quad d_j \geq 0, \quad \textrm{for each} ~j \in Z(x)\right\},
\end{equation}
and $Z(x)$ is given by \eqref{def_Z}. We remark that the Lagrange multipliers $\mu \in \R$ and $s \in \R^n$ corresponding to the constraints $e^T x = 1$ and $x \geq 0$, respectively, are both scaled by $1/2$ and the former is replaced by $x^T Q x$ in \eqref{eq1} by using \eqref{eq2} and \eqref{eq5}.

Note that (\ref{eq1}) -- (\ref{eq5}) are the KKT conditions and any $x \in \Delta_n$ that satisfies these conditions is said to be a {\em KKT point}.

For any KKT point $x \in \R^n$, \eqref{eq6} captures the second order optimality conditions. Note that $\mathbf{D}(x)$ consists of all feasible directions at $x$ that are orthogonal to the gradient of the objective function at $x$. Furthermore, 
\begin{equation} \label{rels_D}
 \mathbf{D}_*(x) \subseteq \mathbf{D}(x) \subseteq  \mathbf{D}^*(x),  
\end{equation}
where 
\begin{eqnarray} 
 \mathbf{D}_*(x) & = & \left\{d \in \R^n: e^T d = 0, \quad d_j = 0, \quad \textrm{for each} ~j \in Z(x)\right\} \label{def_D_inner},\\ 
 \mathbf{D}^*(x) & = & \left\{d \in \R^n: e^T d = 0 \right\}. \label{def_D_outer}  
\end{eqnarray}

\subsection{Global Optimality Conditions}
 
First, we note that the membership problem in ${\cal COP}^n$ can be cast in the form of (StQP) since $Q \in {\cal COP}^n$ if and only if $\nu(Q) \geq 0$. The following theorem establishes that checking the global optimality condition in (StQP) conversely reduces to a membership problem in ${\cal COP}^n$. We include a short proof for the sake of completeness.

\begin{theorem}[Bomze, 1992] \label{nec_suff_stqp}
Let $Q \in \cS^n$ and let $x^* \in \Delta_n$. Then,
\begin{equation} \label{nsc1}
x^* \in \Omega(Q) \quad \textrm{if and only if} \quad Q - \left((x^*)^T Q x^* \right) E \in {\cal COP}^n.
\end{equation}
\end{theorem}
\begin{proof}
Let $x^* \in \Omega(Q)$. Consider $Q^\prime = Q - \left((x^*)^T Q x^* \right) E \in \cS^n$. Then, by Lemma~\ref{shift_invariance}, 
$\nu(Q^\prime) = \nu \left(  Q - \left((x^*)^T Q x^* \right) E \right) = \nu(Q) - \left((x^*)^T Q x^* \right) = \nu(Q) - \nu(Q) = 0$, which implies that $Q^\prime \in {\cal COP}^n$.

Conversely, suppose that $Q - \left( (x^*)^T Q x^* \right) E \in {\cal COP}^n$. Then, for any $x \in \Delta_n$, we have $x^T \left(Q - \left( (x^*)^T Q x^* \right) E \right) x = x^T Q x -  (x^*)^T Q x^* \geq 0$, where we used $x^T E x = (e^Tx)^2 = 1$, which implies that $\nu(Q) = (x^*)^T Q x^*$, i.e., $x^* \in \Omega(Q)$. 
\end{proof}

\section{Standard Quadratic Programs with Exact Doubly Nonnegative  Relaxations} \label{exact_dnn}

In this section, we focus on the set of instances of (StQP) which admit an exact doubly nonnegative relaxation. To that end, let us define
\begin{equation} \label{def_Q}
\cQ^n := \left\{Q \in \cS^n: \ell(Q) = \nu(Q) \right\}.
\end{equation}
We will present alternative characterizations of $\cQ^n$. These characterizations will subsequently be used for identifying several sufficient conditions for membership in $\cQ^n$.

First, given $x \in \Delta_n$, we define the following set of matrices:
\begin{equation} \label{def_Sx}
{\cal S}_x = \left\{Q \in \cS^n: x \in \Omega(Q)\right\} = \left\{Q \in \cS^n: Q - \left(x^T Q x \right) E \in {\cal COP}^n \right\},
\end{equation}
i.e., ${\cal S}_x$ consists of all matrices $Q \in \cS^n$ for which $x \in \Delta_n$ is an optimal solution of the corresponding (StQP) instance. Note that the second equality in \eqref{def_Sx} is a consequence of Theorem~\ref{nec_suff_stqp}. 

Let us define the following line in $\cS^n$, which will frequently arise in the remainder of the paper:
\begin{equation} \label{def_L}
\cL = \left\{\lambda E: \lambda \in \R\right\}.
\end{equation}

For each $x \in \Delta_n$, it is easy to verify that ${\cal S}_x$ is a closed and convex cone in $\cS^n$ and 
\begin{equation} \label{Sx_line}
    \cL \subseteq \cS_x, \quad \textrm{for each}~x \in \Delta_n.
\end{equation}
Furthermore,
\begin{equation} \label{union_Sx}
    \bigcup\limits_{x \in \Delta_n} {\cal S}_x = \cS^n.
\end{equation}

Next, we focus on the characterization of the set of matrices in ${\cal S}_x$ that admit an exact doubly nonnegative relaxation, i.e., 
\begin{equation} \label{def_Qx}
\cQ_x = \cS_x \cap \cQ^n = \left\{Q \in \cS^n: x \in \Omega(Q), \quad \ell(Q) = \nu(Q)\right\}.    
\end{equation}
The following lemma presents a complete characterization of $\cQ_x$.

\begin{lemma} \label{charac_1}
For any $x \in \Delta_n$, 
\begin{equation} \label{Q_c1}
\cQ_x = \left\{Q \in \cS^n: Q - \left(x^T Q x \right) E \in {\cal SPN}^n \right\}.
\end{equation}
\end{lemma}
\begin{proof}
We prove the relation \eqref{Q_c1} by showing that each set is a subset of the other one. Let $x \in \Delta_n$ and let $Q \in \cQ_x$. By \eqref{def_Qx}, $Q \in \cS_x$ and $Q \in \cQ^n$, i.e., $\ell(Q) = \nu(Q) = x^T Q x$. Then, since optimal solutions are attained in (DN-D), there exists $S^* \in {\cal SPN}^n$ such that $\nu(Q)E + S^* = Q$, which implies that $Q - \nu(Q) E = Q - \left(x^T Q x \right) E \in {\cal SPN}^n$.

Conversely, for a given $x \in \Delta_n$, if $Q - \left(x^T Q x \right) E \in {\cal SPN}^n$, then $Q \in \cS_x$ by \eqref{inc_rels} and \eqref{def_Sx}, and $\nu(Q) = x^T Q x$. Furthermore, let $\sigma = x^T Q x$ and $S = Q - \sigma E$. Then, $(\sigma,S)$ is a feasible solution of (DN-D), which implies that $\ell(Q) \geq x^T Q x = \nu(Q)$ since (DN-D) is a maximization problem. Combining this inequality with \eqref{lb_dnn}, we obtain $\ell(Q) = \nu(Q)$, i.e., $Q \in \cQ^n$. We therefore obtain $Q \in \cQ_x$.
\end{proof}

By Lemma~\ref{charac_1}, for any $x \in \Delta_n$ and $Q \in \cS^n$, one can check if $Q \in \cQ_x$ in polynomial time by solving a semidefinite program. Similar to $\cS_x$, it is easy to verify that $\cQ_x$ is a closed convex cone and
\begin{equation} \label{Qx_line}
    \cL \subseteq \cQ_x, \quad \textrm{for each}~x \in \Delta_n,
\end{equation}
where $\cL$ is given by \eqref{def_L}.

Next, for a given $x \in \Delta_n$, we aim to present an alternative and more useful characterization of $\cQ_x$ that would enable us to construct a matrix $Q \in \cQ_x$. To that end, we identify the following subsets, which will be the building blocks for the set $\cQ_x$:
\begin{eqnarray} \label{building_blocks}
\cP_x & = & \left\{P \in {\cal PSD}^n: x^T P x = 0\right\} = \left\{P \in {\cal PSD}^n: P x = 0\right\}
 \label{def_Px}\\
\cN_x & = & \left\{N \in \cN^n: x^T N x = 0 \right\} = \left\{N \in \cN^n: N_{ij} = 0, \quad i \in {A}(x),~j \in {A}(x)\right\}, \label{def_Nx}
\end{eqnarray}
where $A(x)$ is defined as in \eqref{def_P}.

For each $x \in \Delta_n$, note that $\cP_x$ is a face of ${\cal PSD}^n$ and $\cN_x$ is a polyhedral cone in $\cN^n$. Furthermore, for each $P \in \cP_x$ and for each $N \in \cN_x$, we have $P - (x^T P x) E = P \in {\cal SPN}^n$ and $N - (x^T N x) E = N \in {\cal SPN}^n$ by \eqref{inc_rels}. By Lemma~\ref{charac_1}, we therefore obtain
\begin{equation} \label{inc_rel_2}
\cP_x + \cN_x \subseteq \cQ_x \subseteq \cS_x, \quad \textrm{for each}~x \in \Delta_n.
\end{equation}

The next proposition presents a complete characterization of $\cQ_x$ by establishing a useful relation between $\cQ_x$ and the sets $\cN_x$ and $\cP_x$.

\begin{proposition} \label{charac_2}
For each $x \in \Delta_n$, 
\begin{equation} \label{alg_char_Qx}
\cQ_x = \cP_x + \cN_x + \cL,
\end{equation}
where $\cP_x$, $\cN_x$, and $\cL$ are defined as in \eqref{def_Px}, \eqref{def_Nx}, and \eqref{def_L}, respectively. Furthermore, for any decomposition of $Q \in \cQ_x$ given by $Q = P + N + \lambda E$, where $P \in \cP_x$, $N \in \cN_x$, and $\lambda \in \R$, we have $\lambda = x^T Q x = \ell(Q) = \nu(Q)$.
\end{proposition}
\begin{proof}
Let $x \in \Delta_n$ and $Q \in \cQ_x$. Then, by Lemma~\ref{charac_1},
\[
Q - \left(x^T Q x \right) E = P + N,
\]
where $P \in {\cal PSD}^n$ and $N \in \cN^n$. Therefore, 
\[
0 = x^T Q x - \left(x^T Q x \right) \left(x^T E x \right) =  x^T P x + x^T N x,
\]
where we used $x^T E x = (e^T x)^2 = 1$, which implies that $x^T P x = x^T N x = 0$ since both terms are nonnegative. Therefore, we obtain 
\[
Q = P + N + \left(x^T Q x \right) E,
\]
where $P \in \cP_x$ and $N \in \cN_x$. It follows that $Q \in \cP_x + \cN_x + \cL$. 

Conversely, since $\cP_x + \cN_x \subseteq \cQ_x$ by \eqref{inc_rel_2}, $\cL \subseteq \cQ_x$ by \eqref{Qx_line}, and $\cQ_x$ is a convex cone, it follows that $\cP_x + \cN_x + \cL \subseteq \cQ_x$, which establishes \eqref{alg_char_Qx}.

For the last assertion, let $Q \in \cQ_x$ be decomposed as $Q = P + N + \lambda E$, where $P \in \cP_x$, $N \in \cN_x$, and $\lambda \in \R$. Then, $x^T Q x = x^T P x + x^T N x + \lambda$, which implies that $x^T Q x = \lambda$. Since $Q \in \cQ^n$ and $\cQ_x \subseteq \cS_x$, we obtain $\lambda = x^T Q x = \ell(Q) = \nu(Q)$.
\end{proof}

We remark that Proposition~\ref{charac_2} gives a complete characterization of $\cQ_x$ for each $x \in \Delta_n$. In addition, it gives a recipe to construct a matrix in $\cQ_x$. Indeed, for any $x \in \Delta_n$, one simply needs to generate two matrices $P \in \cP_x$, $N \in \cN_x$, a real number $\lambda$, and define $Q = P + N + \lambda E$. By Proposition~\ref{charac_2}, this is necessary and sufficient to ensure that $Q \in \cQ_x$ with $\ell(Q) = \nu(Q) = \lambda$. 

Note that a matrix $P \in \cP_x$ can easily be generated by choosing a matrix $B \in \R^{n \times (n-1)}$ whose columns form a basis for $x^\perp$, and defining $P = B V B^T$, where $V \in {\cal PSD}^{n-1}$. Alternatively, the following discussion illustrates that there is an even simpler procedure to generate such a matrix $P \in \cP_x$, without having to compute a basis for $x^\perp$. To that end, we present a technical result first.

\begin{lemma} \label{rank1update}
For any two vectors $u \in \R^n$ and $v \in \R^n$ such that $u^T v = 1$, we have 
\begin{equation} \label{r1_rel1}
\mathbf{R}(I - uv^T) = v^\perp, 
\end{equation}
where $\mathbf{R}(\cdot)$ denotes the range space.
\end{lemma}
\begin{proof}
Let $w \in \mathbf{R}(I - uv^T)$. Then, there exists $z \in \R^n$ such that $w = (I - u v^T) z = z - (v^T z) u$. Therefore, $v^T w = v^T z - (v^T z) (v^T u) =  v^T z - v^T z = 0$, which implies that $w \in v^\perp$. 

Conversely, if $w \in v^\perp$, then $(I - u v^T) w = w - (v^T w) u = w$, which implies that $w \in \mathbf{R}(I - uv^T)$, establishing \eqref{r1_rel1}. 
\end{proof}

Using Lemma~\ref{rank1update}, we can present a simpler characterization of $\cP_x$.

\begin{lemma} \label{alt_char_Px}
The following identity holds:
\begin{equation} \label{simple_char_Px}
\cP_x = \left\{P \in \cS^n: P = \left(I - e x^T \right) K \left(I - x e^T \right) \textrm{ for some } K \in {\cal PSD}^{n}\right\},
\end{equation}
where $\cP_x$ is given by \eqref{def_Px}.
\end{lemma}
\begin{proof}
Suppose that $P \in \cP_x$. Then, $P \in {\cal PSD}^n$ and $x^T P x = 0$. Since $P \in {\cal PSD}^n$, there exists a matrix $L \in \R^{n \times n}$ such that $P = L L^T$. It follows that $L^T x = 0$, which implies that each column of $L$ belongs to $x^\perp$. Since $e^T x = 1$, it follows from Lemma~\ref{rank1update} that there exists a matrix $W \in \R^n$ such that $L = \left(I - e x^T \right) W$. Therefore, $P = L L^T = \left(I - e x^T \right) W W^T \left(I - x e^T \right) = \left(I - e x^T \right) K \left(I - x e^T \right)$, where $K = W W^T \in {\cal PSD}^{n}$.

Conversely, if $P = \left(I - e x^T \right) K \left(I - x e^T \right)$ for some $K \in {\cal PSD}^{n}$, then we clearly have $P \in {\cal PSD}^n$ and $x^T P x = 0$, which implies that $P \in \cP_x$.
\end{proof}

By Lemma~\ref{alt_char_Px}, in order to ensure that $P \in \cP_x$, it is necessary and sufficient to generate a matrix $K \in {\cal PSD}^{n}$ and define $P = \left(I - ex^T \right) K \left(I - xe^T \right)$.

The following corollary is an immediate consequence of Proposition~\ref{charac_2}, \eqref{def_Qx}, and \eqref{union_Sx}.

\begin{corollary} \label{Q_char_2}
The following relation is satisfied:
\begin{equation} \label{Q_c2}
\cQ^n = \bigcup\limits_{x \in \Delta_n} \cQ_x = \bigcup\limits_{x \in \Delta_n} \left( \cP_x + \cN_x + \cL \right),
\end{equation}
where $\cQ_x$, $\cP_x$, $\cN_x$, and $\cL$ are given by \eqref{def_Qx}, \eqref{def_Px}, \eqref{def_Nx}, and \eqref{def_L}, respectively. 
\end{corollary}

By Lemma~\ref{charac_1}, for any $x \in \Delta_n$ and $Q \in \cS^n$, one can check if $Q \in \cQ_x$ in polynomial time. In contrast, checking if $Q \in \cS_x$ is, in general, NP-hard. Furthermore, a complete characterization of the matrices in $\cS_x \backslash \cQ_x$ requires a full understanding of the set ${\cal COP}^n \backslash {\cal SPN}^n$. While the set of extreme rays of ${\cal COP}^n \backslash {\cal SPN}^n$ has recently been completely characterized for $n = 5$ and $n = 6$ (see \cite{hildebrand2012extreme,afonin:hal-02463284}), the problem still remains open in higher dimensions. 

In the remainder of this section, we establish that the set $\cS_x$ admits simple characterizations under the assumption that $x \in \Delta_n$ satisfies certain conditions. 

To that end, we first recall that the boundary of ${\cal COP}^n$ is given by
\begin{equation} \label{cop_bd}
\partial ~{\cal COP}^n = \left\{M \in {\cal COP}^n: \exists~u \in \Delta_n \textrm{ s.t. } u^T M u = 0 \right\}.
\end{equation} 

For a copositive matrix $M \in \partial ~{\cal COP}^n$, the zeros of $M$ is given by
\begin{equation} \label{def_M_zeros}
\mathbf{V}^M = \left\{u \in \Delta_n: u^T M u = 0\right\}.
\end{equation}

We start with the following simple lemma. We remark that these results can be found in, e.g., ~\cite{ref:Diananda,baumert1966extreme,dickinson2013irreducible}. For the sake of completeness, we provide alternate proofs by relying on the optimality conditions of (StQP).

\begin{lemma} \label{tech_res_cop}
Let $Q \in \cS^n$ and let $x^* \in \Omega(Q)$. Let $M = Q - \left( (x^*)^T Q x^* \right) E \in \cS^n$, $A = A(x^*)$ and $Z = Z(x^*)$, where $A(\cdot)$ and $Z(\cdot)$ are defined as in \eqref{def_P} and \eqref{def_Z}, respectively. Then, 
\begin{enumerate}
    \item[(i)] $M \in \partial ~{\cal COP}^n$;
    \item[(ii)] $M_{AA} \, x^*_A = 0$;
    \item[(iii)] $M_{ZA} \, x^*_A \geq 0$;
    \item[(iv)] $M_{AA} \in {\cal PSD}^{|A|}$.
\end{enumerate}
\end{lemma}
\begin{proof}
Let $Q \in \cS^n$, $x^* \in \Omega(Q)$, and $M = Q - \left( (x^*)^T Q x^* \right)E$. By Theorem~\ref{nec_suff_stqp}, $M \in {\cal COP}^n$. Furthermore, $(x^*)^T M x^* = (x^*)^T Q x^* - (x^*)^T Q x^* = 0$, which implies that $M \in \partial ~{\cal COP}^n$ by \eqref{cop_bd}, which establishes (i). 

Consider the (StQP) instance corresponding to $M$. Since $M \in {\cal COP}^n$, we obtain $x^T M x \geq (x^*)^T M x^* = 0 = \nu(M)$ for each $x \in \Delta_n$. By combining $\nu(M) = (x^*)^T M x^* = 0$ with the KKT conditions \eqref{eq1}, \eqref{eq4}, and \eqref{eq5}, we obtain $M_{AA} \, x^*_A = 0$ and $M_{ZA} \, x^*_A \geq 0$, establishing (ii) and (iii). 

Finally, for any $d \in \R^{|A|}$ and any $\alpha \in \R$, we have
\[
(x^*_A + \alpha d)^T M_{AA} (x^*_A + \alpha d) = (x^*_A)^T M_{AA} x^*_A + 2 \alpha d^T M_{AA} x^*_P + \alpha^2 d^T M_{AA} d = \alpha^2 d^T M_{AA} d,
\]
where we used $(x^*)^T M x^* = (x^*_A)^T M_{AA} x^*_A = 0$ and (ii) in the second equality. If there exists $d \in \R^{|A|}$ such that $d^T M_{AA} d < 0$, then, since $x^*_A > 0$, for sufficiently small $\alpha > 0$, we obtain $x^*_A + \alpha d > 0$ and $(x^*_A + \alpha d)^T M_{AA} (x^*_A + \alpha d) < 0$, which implies that $M_{AA} \not \in {\cal COP}^n$, contradicting Lemma~\ref{lem_gen_rels} (iv). Therefore, $M_{AA} \in {\cal PSD}^{|A|}$, establishing (iv).
\end{proof}

We are now in a position to identify some points $x \in \Delta_n$ for which the set $\cS_x$ given by \eqref{def_Sx} has a simple description.

\begin{lemma} \label{large_support}
For any $x \in \Delta_n$ such that $|A(x)| \geq n - 1$, where $A(x)$ is given by \eqref{def_Px}, we have
\begin{equation} \label{Sx_eq_Qx}
\cS_x = \cQ_x,    
\end{equation}
where $\cS_x$ and $\cQ_x$ are given by \eqref{def_Sx} and \eqref{def_Qx}, respectively. 
\end{lemma}
\begin{proof}
Let $x \in \Delta_n$ be such that $|A(x)| \geq n - 1$. 
Note that we already have $\cQ_x \subseteq \cS_x$ by \eqref{def_Qx}. Therefore, it suffices to establish the reverse inclusion. 

Let $Q \in \cS_x$ and let $M = Q - \left( x^T Q x \right) E \in \cS^n$. By Lemma~\ref{tech_res_cop} (i), we have $M \in \partial ~{\cal COP}^n$. Let $A = A(x)$. If $|A| = n$, then $M_{AA} = M \in {\cal PSD}^{n}$ by Lemma~\ref{tech_res_cop} (iv), which implies that $M \in {\cal SPN}^n$ by \eqref{inc_rels} and $Q \in \cQ_x$ by Lemma~\ref{charac_1}. If, on the other hand, $|A| = n - 1$, then $M_{AA} \in {\cal PSD}^{n-1}$ by Lemma~\ref{tech_res_cop} (iv). By \cite[Lemma 3.1]{shaked2016spn}, it follows that $M \in {\cal SPN}^n$ and we similarly obtain $Q \in \cQ_x$.
\end{proof}

Our final result specifically focuses on the case $n = 5$. 

\begin{lemma} \label{n_5_singleton}
Let $n = 5$. Then, for $x \in \left\{e_1,e_2,\ldots,e_5\right\}$, we have
\begin{equation} \label{Sx_eq_Qx_n_5}
\cS_x = \cQ_x,    
\end{equation}
where $\cS_x$ and $\cQ_x$ are given by \eqref{def_Sx} and \eqref{def_Qx}, respectively.
\end{lemma}
\begin{proof}
Let $Q \in \cS_x$, where $x \in \left\{e_1,e_2,\ldots,e_5\right\}$, and let $M = Q - \left( x^T Q x \right) E \in \cS^5$. Then, $|A(x)| = |A| = 1$. By Lemma~\ref{tech_res_cop} (ii) and (iii), there exists a permutation matrix $J \in \R^{n \times n}$ such that
\[
J^T M J = \widehat{M} = \begin{bmatrix} 0 & b^T \\ b & B \end{bmatrix},
\]
where $b \in \R^4_+$ and $B \in \cS^4$. Since $M \in {\cal COP}^5$, we have $B \in {\cal COP}^4$ by Lemma~\ref{lem_gen_rels} (ii) and (iv). By \eqref{diananda}, $B \in {\cal SPN}^4$. Since $b \geq 0$, it follows from \cite[Lemma 3.3]{shaked2016spn} that $\widehat{M} \in {\cal SPN}^5$, which implies that $M \in {\cal SPN}^5$ by Lemma~\ref{lem_gen_rels} (ii) and that $Q \in \cQ_x$ by Lemma~\ref{charac_1}.
\end{proof}

For any $x \in \Delta_n$ that satisfies the conditions of Lemma~\ref{large_support} or Lemma~\ref{n_5_singleton}, it follows that the doubly nonnegative relaxation is exact for all instances of (StQP) for which $x$ is an optimal solution. We also remark that the proof of Lemma~\ref{n_5_singleton} cannot be extended to the case $n \geq 6$. In fact, for any $n \geq 6$, we will illustrate in Section~\ref{pos_gap} how to construct an instance of (StQP) with $\left\{e_1,e_2,\ldots,e_n\right\} \subseteq \Omega(Q)$ such that the doubly nonnegative relaxation has a positive gap.

We close this section by recalling that, for each $x \in \Delta_n$, the membership problem in $\cQ_x$ is polynomial-time solvable. On the other hand, for a given $Q \in \cS^n$, checking if $Q \in \cQ^n$ is equivalent to checking if there exists $x \in \Delta_n$ such that $Q \in \cQ_x$. Since this latter problem may not necessarily be polynomial-time solvable, we instead focus on explicitly identifying several classes of matrices that belong to $\cQ^n$ in the next section.

\section{Three Families of Standard Quadratic Programs with Exact Doubly Nonnegative Relaxations} \label{subsets_of_cQ}

In this section, we identify three families of matrices that admit exact doubly nonnegative relaxations by relying on the characterizations presented in Section~\ref{exact_dnn}.

\subsection{Minimum Entry on the Diagonal} \label{min_diag}

In this section, we show that any matrix $Q \in \cS^n$ whose minimum entry lies on the diagonal belongs to $\cQ^n$. Let us denote the set of such matrices by $\cQ^n_{1}$, i.e.,
\begin{equation} \label{def_cQ1}
\cQ^n_1 = \left\{Q \in \cS^n: \min\limits_{1 \leq i \leq j \leq n} Q_{ij} = \min\limits_{k = 1,\ldots,n} Q_{kk}\right\}.    
\end{equation}

Note that $\cQ^n_1$ is given by the union of a finite number of polyhedral cones, i.e.,
\[
\cQ^n_1 = \bigcup_{k=1}^n \left\{Q \in \cS^n: Q_{ij} \geq Q_{kk}, \quad 1 \leq i \leq j \leq n\right\}.
\]

\begin{proposition} \label{Q1subsetQ}
The following relation holds:
\begin{equation} \label{Q1_subset_Q}
\cQ^n_1 \subseteq \cQ^n,    
\end{equation}
where $\cQ^n_1$ and $\cQ^n$ are given by \eqref{def_cQ1} and \eqref{def_Q}, respectively.
\end{proposition}
\begin{proof}
Let $Q \in \cQ^n_1$. Let us define $\lambda = \min\limits_{1 \leq i \leq j \leq n} Q_{ij} = \min\limits_{k = 1,\ldots,n} Q_{kk} = Q_{\ell \ell}$ and $N = Q - \lambda E \in \cN^n$. Therefore, $Q = 0 + N + \lambda E$. Then, it easy to verify that $N \in \cN_x$, where $x = e_{\ell} \in \R^n$ and $\cN_x$ is given by \eqref{def_Nx}. By Proposition~\ref{charac_2}, $Q \in \cQ_x$, where $\cQ_x$ is given by \eqref{def_Qx}. The inclusion \eqref{Q1_subset_Q} follows. 
\end{proof}

\subsubsection{Standard Quadratic Programs with a Concave Objective Function}

In this section, we explicitly identify a subset of matrices contained in $\cQ^n_1$, where $\cQ^n_1$ is given by \eqref{def_cQ1}.

By the proof of Proposition~\ref{Q1subsetQ}, 
\begin{equation} \label{Q1_implication}
\Omega(Q) \cap \{e_1,\ldots,e_n\} \neq \emptyset, \quad \forall Q \in \cQ^n_1.
\end{equation}

Based on this observation, it is worth focusing on the set of instances of (StQP) with a concave objective function since the set of optimal solutions necessarily contains one of the vertices of the unit simplex. Such instances are precisely given by those instances in which $Q$ is negative semidefinite on $e^\perp$, i.e.,
\begin{equation} \label{def_Q_concave}
\cQ^n_{\textrm{concave}} = \left\{ Q \in \cS^n: d^T Q d \leq 0, \quad \forall d \in \R^n ~\textrm{such that}~ e^T d = 0\right\}.
\end{equation}

The following inclusion can easily be verified.
\begin{equation} \label{nsd_inclusion}
- {\cal PSD}^n + \cL \subseteq \cQ^n_{\textrm{concave}},
\end{equation}
where $\cL$ is given by \eqref{def_L}.

First, we present a useful property of $\cQ^n_{\textrm{concave}}$.

\begin{lemma} \label{alt_char_cQconcave}
For any $Q \in \cQ^n_{\textrm{concave}}$,  
\begin{equation} \label{alt_char_nsd}
- \left(I - e x^T \right) Q \left(I - x e^T\right) \in {\cal PSD}^n, \quad \textrm{for each}~x \in \Delta_n.
\end{equation}
\end{lemma}
\begin{proof}
The assertion follows directly from Lemma~\ref{rank1update} since $e^T x = 1$ for each $x \in \Delta_n$.
\end{proof}

We next show that every matrix $Q \in \cQ^n_{\textrm{concave}}$ necessarily has a minimum entry along the diagonal.

\begin{proposition} \label{Q_concave_subset_Q1}
The following relation holds:
\begin{equation} \label{Q_conc_subset_Q1}
\cQ^n_{\textrm{concave}} \subseteq \cQ^n_1,    
\end{equation}
where $\cQ^n_{\textrm{concave}}$ and $\cQ^n_1$ are given by \eqref{def_Q_concave} and \eqref{def_cQ1}, respectively. Therefore, 
\begin{equation} \label{Q_concave_subset_Q}
\cQ^n_{\textrm{concave}} \subseteq \cQ^n.    
\end{equation}
\end{proposition}

\begin{proof}
Suppose, for a contradiction, that \eqref{Q_conc_subset_Q1} does not hold. Then, there exists $Q \in \cQ^n_{\textrm{concave}}$ such that $Q \not \in \cQ^n_1$, i.e., there exists a tuple $(k,l)$ such that $1 \le k < l \le n$ and
\be \label{kl}
\min\limits_{1 \leq i \leq j \leq n} Q_{ij} = Q_{kl} <  \min\limits_{i = 1,\ldots,n} Q_{ii}.
\ee

Now, let us define 
\begin{equation} \label{def_Y}
Y:= - \left(I- (1/n)E\right) Q \left(I-(1/n)E\right).
\end{equation}
By Lemma~\ref{alt_char_cQconcave}, $Y \in {\cal PSD}^n$, which implies that 
\begin{equation} \label{psd_ineq}
Y_{kk} + Y_{ll} \geq 2 Y_{kl}.
\end{equation}
By \eqref{def_Y}, 
\begin{eqnarray*}
Y_{kk} & = & - Q_{kk} + \frac{2}{n} e^T Q e_k - \frac{1}{n^2} e^T Q e,\\
Y_{ll} & = & - Q_{ll} + \frac{2}{n} e^T Q e_l - \frac{1}{n^2} e^T Q e,\\
Y_{kl} & = & - Q_{kl} + \frac{1}{n} e^T Q e_l + \frac{1}{n} e^T Q e_k - \frac{1}{n^2} e^T Q e,
\end{eqnarray*}
which, together with \eqref{psd_ineq}, implies that
\[
Q_{kk} + Q_{ll} \leq 2 Q_{kl},
\]
contradicting \eqref{kl}. The relation \eqref{Q_conc_subset_Q1} follows. The inclusion \eqref{Q_concave_subset_Q} is an immediate consequence of Proposition~\ref{Q1subsetQ}.

\end{proof}

We close this section by making two observations. First, we remark that the inclusion \eqref{Q_conc_subset_Q1} can be strict since, for instance, 
\[
Q = \begin{bmatrix} 0 & 0 \\ 0 & 1 \end{bmatrix} \in \cQ^2_1 \backslash \cQ^2_{\textrm{concave}}
\]
since, for $d = [-1,1]^T$, we have $e^T d = 0$ but $d^T Q d > 0$.

Second, we illustrate, by an example, that the set of matrices that satisfy the relation \eqref{Q1_implication} is strictly larger than $\cQ^n_1$. For instance,
\[
Q = \begin{bmatrix} 0 & 0 & 0 \\ 0 & 2 & -1 \\ 0 & -1 & 2 \end{bmatrix} \not \in \cQ^3_1,
\]
whereas $\Omega(Q) = \{e_1\}$. On the other hand, for $n = 5$, recall that 
the condition \eqref{Q1_implication} is sufficient to ensure that $Q \in \cQ^n$ by Lemma~\ref{n_5_singleton}.

\subsection{Standard Quadratic Programs with a Convex Objective Function} \label{psd_on_e_perp}

In this section, we focus on instances of (StQP) whose objective function is convex over $\Delta_n$. Note that such instances are precisely characterized by matrices $Q \in \cS^n$ that are positive semidefinite on $e^\perp$, i.e., 
\begin{equation} \label{psd_e_perp}
d^T Q d \geq 0, \quad \forall d \in \R^n ~\textrm{such that}~ e^T d = 0.
\end{equation}

Let us accordingly define the following set:
\begin{equation} \label{def_cQ2}
\cQ^n_2 = \left\{Q \in \cS^n: d^T Q d \geq 0, \quad \forall d \in \R^n ~\textrm{such that}~ e^T d = 0 \right\}.
\end{equation}

Clearly, we have
\begin{equation} \label{psd_inclusion}
{\cal PSD}^n + \cL \subseteq \cQ^n_2,
\end{equation}
where $\cL$ is given by \eqref{def_L}. For any $Q \in \cQ^n_2$, consider the corresponding (StQP) instance. It follows from \eqref{rels_D} and \eqref{eq1}--\eqref{eq6} that any KKT point is a local minimizer. In fact, by the convexity of the objective function over the feasible region, any KKT point is, in fact, a global minimizer.

In this section, we aim to establish that $\cQ^n_2 \subseteq \cQ^n$. First, we present a technical result that is similar to Lemma~\ref{alt_char_cQconcave}, which would be useful to prove this inclusion.

\begin{lemma} \label{alt_char_cQ2}
For any $Q \in \cQ^n_2$,  
\begin{equation} \label{alt_char_psd}
\left(I - e x^T \right) Q \left(I - x e^T\right) \in {\cal PSD}^n, \quad \textrm{for each}~x \in \Delta_n.
\end{equation}
\end{lemma}
\begin{proof}
The assertion follows directly from Lemma~\ref{rank1update} since $e^T x = 1$ for each $x \in \Delta_n$.
\end{proof}

Next, we present our main result in this section.

\begin{proposition} \label{Q2subsetQ}
The following relation holds:
\begin{equation} \label{Q2_subset_Q}
\cQ^n_2 \subseteq \cQ^n,    
\end{equation}
where $\cQ^n_2$ and $\cQ^n$ are given by \eqref{def_cQ2} and \eqref{def_Q}, respectively.
\end{proposition}
\begin{proof}
If $Q \in {\cal PSD}^n$, then $Q \in \cQ^n$ by \cite[Lemma 2.7]{KimKT2000}. Otherwise,  
let $Q \in \cQ^n_2$ and $x \in \Omega(Q)$. It suffices to show that $Q \in \cQ_x$, where $\cQ_x$ is given by \eqref{def_Qx}. By Proposition~\ref{charac_2}, we need to construct a decomposition
\[
Q = P + N + \left( x^T Q x \right) E,
\]
where $P \in \cP_x$, $N \in \cN_x$, and $\cP_x$ and $\cN_x$ are given by \eqref{def_Px} and \eqref{def_Nx}, respectively.

Let us define
\[
P = \left(I - e x^T \right) Q \left(I - x e^T \right).
\]
By Lemma~\ref{alt_char_cQ2}, $P \in {\cal PSD}^n$. Therefore, 
\[
P = Q - Q x e^T - e x^T Q + \left(x^T Q x \right) E,
\]
or equivalently, 
\[
Q - \left(x^T Q x \right) E = P + \left( Q x e^T + e x^T Q - 2 \left( x^T Q x \right) E \right).
\]
Let us accordingly define
\[
N = Q x e^T + e x^T Q - 2 \left( x^T Q x \right) E.
\]
It suffices to show that $N \in \cN_x$. 
Since $x \in \Omega(Q)$, $x$ is a KKT point, i.e., there exists $s \in \R^n$ such that the conditions \eqref{eq1} -- \eqref{eq5} are satisfied. By~\eqref{eq1},
\[
Qx - \left( x^T Q x \right) e - s = 0,
\]
which implies that
\begin{eqnarray*}
Q x e^T - \left( x^T Q x \right) E - s e^T & = & 0, \\
e x^T Q - \left( x^T Q x \right) E - e s^T & = & 0.
\end{eqnarray*}
It follows from these two equations that
\begin{eqnarray*}
N & = & Q x e^T + e x^T Q - 2 \left( x^T Q x \right) E \\
 & = & s e^T + \left( x^T Q x \right) E + e s^T + \left( x^T Q x \right) E - 2 \left( x^T Q x \right) E \\
 & = & s e^T + e s^T.
\end{eqnarray*}
Finally, note that $N \in \cN_x$ since $N \in \cN^n$ and $x^T N x = 0$ by \eqref{eq2}, \eqref{eq4}, and \eqref{eq5}. It follows from Proposition~\ref{charac_2} that $Q \in \cQ_x$.
\end{proof}

Note that the proof of Proposition~\ref{Q2subsetQ} is based on an explicit construction of the decomposition of a matrix $Q \in \cQ^n_2$ given by Proposition~\ref{charac_2}. 

We close this section by the following observation. By Proposition~\ref{Q2subsetQ}, we have $\cQ^n_2 \subseteq \cQ^n$, where $\cQ^n_2$ and $\cQ^n$ are given by \eqref{def_cQ2} and \eqref{def_Q}, respectively. Clearly, $\cQ^n_2 = - \cQ^n_{\textrm{concave}}$ by \eqref{def_cQ2} and \eqref{def_Q_concave}. Therefore, for each $Q \in \cQ^n_2 \backslash \{0\}$, it follows from Propositions~\ref{Q_concave_subset_Q1} and \ref{Q2subsetQ} that each of the hyperplanes 
\begin{eqnarray*}
\cH_1 & = & \left\{Y \in \cS^n: \langle Q, Y \rangle = \ell(Q) = \nu(Q) \right\}, \\
\cH_2 & = & \left\{Y \in \cS^n: \langle -Q, Y \rangle = \ell(-Q) = \nu(-Q) \right\},
\end{eqnarray*}
is a supporting hyperplane of both of the feasible regions of (DN-P) and (CP).

\subsection{Maximum Weighted Cliques on Perfect Graphs} \label{max_w_clique}

In this section, we identify another family of instances of (StQP) that admits an exact doubly nonnegative relaxation.

First, we briefly review the maximum weighted clique problem. Let $G = (V_G,E_G)$ be a simple, undirected graph with $V_G = \{1,\ldots,n\}$ and let $w \in \R^n_{++}$, where $w_k$ denotes the weight of vertex $k,~k = 1,\ldots,n$. A set $C \subseteq V_G$ is a clique if all pairs of vertices in $C$ are connected by an edge. The weight of a clique $C \subseteq V_G$, denoted by $w(C)$, is given by $w(C) = \sum\limits_{j \in C} w_j$. The maximum weighted clique problem is concerned with finding a clique with the maximum weight, and its weight is denoted by $\omega(G,w)$. Note that the maximum weighted clique problem is equivalent to the maximum clique problem if all the weights are identical.

For a given graph $G = (V_G,E_G)$, the complement of $G$, denoted by $\overbar{G}$, is the graph on $V_G$ obtained by deleting all edges in $E_G$ and connecting each pair of nonadjacent vertices in $G$. For a set $V \subseteq V_G$, the subgraph of $G$ induced by $V$ is the graph whose vertices are given by $V$ and whose edges are given by the edges in $E_G$ with both endpoints in $V$. The maximum weighted clique problem is therefore concerned with finding an induced complete subgraph in $G$ with the maximum weight. Recall that $G$ is a perfect graph if neither $G$ nor its complement contains an odd cycle of length at least five as an induced subgraph~\cite{chudnovsky2006strong}. 

We next discuss the connection between the maximum weighted clique problem and (StQP). Let $G = (V_G,E_G)$ be a graph with $V_G = \{1,\ldots,n\}$ and let $w \in \R^n_{++}$, where $w_k$ denotes the weight of vertex $k,~k = 1,\ldots,n$. Let us define the following class of matrices:
\begin{equation} \label{def_M_class}
{\cal M}(G,w) = \left\{B \in {\cal S}^n: \begin{array}{ll}B_{kk} = 1/w_k, & k = 1,\ldots,n,\\ B_{ij} = 0 , & (i,j) \in E_G,\\ 2 B_{ij} \geq B_{ii} + B_{jj}, &  (i,j) \in E_{\overbar{G}} \end{array} \right\}.
\end{equation}

The following theorem establishes the aforementioned connection.

\begin{theorem}[Gibbons et al., 1997] \label{mwss_stqp}
Let $G = (V_G,E_G)$ be a graph with $V_G = \{1,\ldots,n\}$ and let $w \in \R^n_{++}$, where $w_k$ denotes the weight of vertex $k,~k = 1,\ldots,n$. Then, for any $Q \in {\cal M}(G,w)$,
\begin{equation} \label{M_class_nu}
\nu(Q) = \min\{x^T Q x: x \in \Delta_n\} = \frac{1}{\omega(G,w)}.
\end{equation}
\end{theorem}

Theorem~\ref{mwss_stqp} is a generalization of the well-known Motzkin-Straus Theorem~\cite{motzkin1965maxima} that establishes the first connection between the (unweighted) maximum clique problem and a particular instance of (StQP) associated with the underlying graph.

We next discuss the weighted Lov{\'a}sz theta number. Let $G = (V_G,E_G)$ be a graph with $V_G = \{1,\ldots,n\}$ and let $w \in \R^n_{++}$, where $w_k$ denotes the weight of vertex $k,~k = 1,\ldots,n$. The weighted Lov{\'a}sz theta number~\cite{ref:lovasz1979,ref:grotschel1981} corresponding to the complement graph $\overbar{G}$ is given by
\begin{equation} \label{w_lov}
\vartheta(\overbar{G},w) = \max\left\{\langle W,X \rangle: \langle I, X \rangle = 1, \quad X_{ij} = 0,~(i,j) \in E_{\overbar{G}}, \quad X \in {\cal PSD}^n\right\},
\end{equation}
where $W \in \cS^n$ is given by 
\begin{equation} \label{def_W}
W_{ij} = \sqrt{w_i w_j},~1 \leq i \leq j \leq n.
\end{equation}

The weighted Lov{\'a}sz theta number satisfies $\omega(G,w) \leq \vartheta(\overbar G,w)$~\cite{ref:lovasz1979,ref:grotschel1981}. Furthermore,  
\begin{equation} \label{lov_per}
\omega(G,w) = \vartheta(\overbar{G},w) \quad \textrm{if $G$ is a perfect graph}.
\end{equation}

The weighted Lov{\'a}sz theta number can be strengthened by replacing the constraint $X \in {\cal PSD}^n$ by $X \in {\cal DN}^n$~\cite{schrijver1979comparison}:
\begin{equation} \label{w_lov_sch}
\vartheta^\prime(\overbar{G},w) = \max\left\{\langle W,X \rangle: \langle I, X \rangle = 1, \quad X_{ij} = 0,~(i,j) \in E_{\overbar{G}}, \quad X \in {\cal DN}^n\right\},
\end{equation}
The strengthened version of the weighted Lov{\'a}sz theta number satisfies the following relations:
\begin{equation} \label{sandwich}
\omega(G,w) \leq \vartheta^\prime(\overbar{G},w) \leq \vartheta(\overbar{G},w).
\end{equation}
By \eqref{lov_per} and \eqref{sandwich}, 
\begin{equation} \label{lov_sch_per}
\omega(G,w) = \vartheta^\prime(\overbar{G},w) \quad \textrm{if $G$ is a perfect graph}.
\end{equation}

For any $w \in \R^n_{++}$ and any $G = (V_G,E_G)$, where $V_G = \{1,\ldots,n\}$, , we next establish that the strengthened version of the weighted Lov{\'a}sz theta number given by \eqref{w_lov_sch} coincides with the reciprocal of the lower bound arising from the doubly nonnegative relaxation of the (StQP) instance corresponding to any $Q \in {\cal M}(G,w)$, i.e., for any $w \in \R^n_{++}$ and any $Q \in {\cal M}(G,w)$,
\[
\ell(Q) = \min\left\{\langle Q, X \rangle: \langle E, X \rangle = 1, \quad X \in {\cal DN}^n \right\} = \frac{1}{\vartheta^\prime(\overbar{G},w)}.
\]

First, we prove a useful property of the doubly nonnegative relaxation.

\begin{lemma} \label{DNN_optimal_exist}
Let $G = (V_G,E_G)$ be simple, undirected graph with $V_G = \{1,\ldots,n\}$. For any $w \in \R^n_{++}$ and any $Q \in {\cal M}(G,w)$, where ${\cal M}(G,w)$ is given by \eqref{def_M_class}, there exists an optimal solution $X^* \in \cS^n$ of (DN-P) such that
\[
X^*_{ij} = 0, \quad \forall (i,j) \in E_{\overbar{G}}.
\]
\end{lemma}
\begin{proof}
Let $Q \in {\cal M}(G,w)$ and $X^* \in {\cal DN}^n$ be an optimal solution of (DN-P). Suppose that $X^*_{ij} > 0$ for some $(i,j) \in E_{\overbar{G}}$. Let us define
\[
X(\alpha) := X^* + \alpha (e_i - e_j) (e_i - e_j)^T.
\]
Observe that $X(\alpha) \in {\cal DN}^n$ for any $0 \le \alpha \le X^*_{ij}$. Furthermore,
\[
\langle Q, X(\alpha) \rangle = \langle Q, X^* \rangle + \alpha\underbrace{(Q_{ii} + Q_{jj} - 2Q_{ij})}_{\le 0} \le \langle Q, X^* \rangle,
\]
where the inequality follows from \eqref{def_M_class}. By setting $\alpha = X^*_{ij}$ and repeating this procedure for any other edges in $E_{\overbar{G}}$ if necessary, we obtain an optimal solution with the desired property.
\end{proof}

We are now in a position to establish the aforementioned relation.

\begin{proposition} \label{dnn_char_M}
Let $G = (V_G,E_G)$ be simple, undirected graph with $V_G = \{1,\ldots,n\}$. For any $w \in \R^n_{++}$ and any $Q \in {\cal M}(G,w)$,  where ${\cal M}(G,w)$ is given by \eqref{def_M_class},
\begin{equation} \label{M_class_ell}
\ell(Q) = \frac{1}{\vartheta^\prime(\overbar{G},w)},
\end{equation}
where $\vartheta^\prime(\overbar{G},w)$ is given by \eqref{w_lov_sch}.
\end{proposition}
\begin{proof}
Let $w \in \R^n_{++}$ and let $Q \in {\cal M}(G,w)$.

First, we will show that $\displaystyle \ell(Q) \le 1/\vartheta^\prime(\overbar{G},w)$. Note that an optimal solution $X^*_{LS} \in {\cal DN}^n$ of \eqref{w_lov_sch} exists since the feasible region is nonempty and compact. Furthermore, $\vartheta^\prime(\overbar G,w) = \langle W, X^*_{LS} \rangle > 0$ since $W$ has strictly positive components and $\langle I, X^*_{LS} \rangle = 1$. Let us define $\hat{X}_{LS} = D_{LS} X^*_{LS} D_{LS}$, where $D_{LS} \in \cS^n$ is a diagonal matrix given by 
\begin{equation} \label{def_D_LS}
D_{LS} = \frac{1}{\sqrt{\vartheta^\prime(\overbar{G},w)}} \begin{bmatrix} \sqrt{w_1} & & \\ & \ddots & \\ & & \sqrt{w_n} \end{bmatrix}.
\end{equation}
By Lemma~\ref{lem_gen_rels} (iii), $\hat{X}_{LS} \in {\cal DN}^n$. Furthermore, 
\[
\vartheta^\prime(\overbar{G},w) = \langle W, X^*_{LS} \rangle = \langle D_{LS}^{-1} W D_{LS}^{-1}, \hat{X}_{LS} \rangle = \vartheta^\prime(\overbar{G},w) \langle E, \hat{X}_{LS} \rangle,
\]
where we used \eqref{def_W} to derive the third equality. Therefore, $\langle E, \hat{X}_{LS} \rangle = 1$, i.e., $\hat{X}_{LS}$ is a feasible solution of (DN-P). By \eqref{def_M_class}, for any $Q \in {\cal M}(G,w)$, 
\[
D_{LS} Q D_{LS} = \frac{1}{\vartheta^\prime(\overbar{G},w)} \left(I + N_{LS} \right),
\]
where $N_{LS} \in \cN^n$ and $(N_{LS})_{ij} = Q_{ij} = 0$ for each $(i,j) \in E_G$. Therefore, 
\begin{eqnarray*}
\langle Q, \hat{X}_{LS} \rangle & = & \langle D_{LS} Q D_{LS}, D_{LS}^{-1} \hat{X}_{LS} D_{LS}^{-1} \rangle \\
& = & \langle D_{LS} Q D_{LS}, X^*_{LS} \rangle \\
& = & \frac{1}{\vartheta^\prime(\overbar{G},w)} \langle I + N_{LS}, X^*_{LS} \rangle \\
& = & \frac{1}{\vartheta^\prime(\overbar{G},w)} \left(1 + \langle N_{LS}, X^*_{LS} \rangle\right),
\end{eqnarray*}
where we used $\langle I,  X^*_{LS} \rangle = 1$ in the last line. Since $Q_{ij} = (N_{LS})_{ij} = 0$ for each $(i,j) \in E_G$ and $(X^*_{LS})_{ij} = 0$ for each $(i,j) \in E_{\overbar{G}}$, it follows that $\langle N_{LS}, X^*_{LS} \rangle = 0$, which implies that $\langle Q, \hat{X}_{LS} \rangle = 1/\vartheta^\prime(\overbar{G},w)$. Therefore, $\ell(Q) \leq 1/\vartheta^\prime(\overbar{G},w)$ since $\hat{X}_{LS}$ is a feasible solution of (DN-P).

Conversely, let $X^*_{DN} \in {\cal DN}^n$ be an optimal solution of (DN-P). Then, $\langle E, X^*_{DN} \rangle = 1$ and $\ell(Q) = \langle Q , X^*_{DN} \rangle > 0$ since $Q \in \cN^n$ with strictly positive diagonal entries. By Lemma~\ref{DNN_optimal_exist}, we can assume that $(X^*_{DN})_{ij} = 0$ for each $(i,j) \in E_{\overbar{G}}$. Let us define another diagonal matrix $D_{DN} \in \cS^n$ given by
\begin{equation} \label{def_D_DN}
D_{DN} = \frac{1}{\sqrt{\ell(Q)}} \begin{bmatrix} \frac{1}{\sqrt{w_1}} & & \\ & \ddots & \\ & & \frac{1}{\sqrt{w_n}} \end{bmatrix}.
\end{equation}
Let $\hat{X}_{DN} = D_{DN} X^*_{DN} D_{DN}$. Once again, by Lemma~\ref{lem_gen_rels} (iii), $\hat{X}_{DN} \in {\cal DN}^n$. Similarly, by \eqref{def_M_class}, for any $Q \in {\cal M}(G,w)$, 
\[
D_{DN}^{-1} Q D_{DN}^{-1} = \ell(Q) \left(I + N_{DN} \right),
\]
where $N_{DN} \in \cN^n$ and $(N_{DN})_{ij} = Q_{ij} = 0$ for each $(i,j) \in E_G$. Therefore, 
\[
\ell(Q) = \langle Q , X^*_{DN} \rangle = \langle D_{DN}^{-1} Q D_{DN}^{-1}, \hat{X}_{DN} \rangle = \ell(Q) \left(\langle I, \hat{X}_{DN} \rangle + \langle N_{DN}, \hat{X}_{DN} \rangle \right).
\]
Note that $\langle N_{DN}, \hat{X}_{DN} \rangle = 0$ since $(N_{DN})_{ij} = Q_{ij} = 0$ for each $(i,j) \in E_G$ and $(\hat{X}_{DN})_{ij} = (X^*_{DN})_{ij} = 0$ for each $(i,j) \in E_{\overbar{G}}$. It follows that $\langle I, \hat{X}_{DN} \rangle = 1$, which, combined with the previous observation, implies that $\hat{X}_{DN}$ is feasible for \eqref{w_lov_sch}. 

Finally, we have
\[
\langle W, \hat{X}_{DN} \rangle = \langle D_{DN} W D_{DN}, X^*_{DN} \rangle = \frac{1}{\ell(Q)} \langle E, X^*_{DN} \rangle = \frac{1}{\ell(Q)},
\]
which implies that $\vartheta^\prime(\overbar{G},w) \geq 1/\ell(Q)$, establishing the reverse inequality. The relation \eqref{M_class_ell} follows.

 \end{proof}
 
 For a given graph $G = (V_G,E_G)$, consider the unweighted case, i.e., let $w = e \in \R^n$ and let $Q = I + A_{\overbar{G}} \in {\cal M}(G,e)$, where $A_{\overbar{G}} \in \cS^n$ is the vertex adjacency matrix of $\overbar{G}$. We remark that the identity $\ell(Q) = 1/\vartheta^\prime(\overbar{G},e)$ is a consequence of Corollary 2.4 and Lemma 5.2 in~\cite{de2002approximation}. It follows that Proposition~\ref{dnn_char_M} generalizes this identity to the weighted case and to any $Q \in {\cal M}(G,w)$.
 
 We now have all the ingredients to establish our main result in this section. Let us first introduce the set of all perfect graphs on the set of vertices $\{1,2,\ldots,n\}$, i.e.,
 \begin{equation}
\mathfrak{G} = \left\{G = (V_G,E_G): V_G = \{1,2,\ldots,n\}, \quad \textrm{$G$ is a perfect graph}\right\}.     
 \end{equation}
 
We next define the following set.
\begin{equation} \label{def_cM}
\cM = \bigcup_{G \in \mathfrak{G}} \bigcup_{w \in \R^n_{++}} {\cal M}(G,w).     
\end{equation}
 
For each perfect graph $G = (V_G,E_G) \in \mathfrak{G}$ and each  $w \in \R^n_{++}$, note that ${\cal M}(G,w)$ is a polyhedral set. Therefore, $\cM$ is given by the union of an infinite number of polyhedral sets.

Finally, we define
\begin{equation} \label{def_Q3}
\cQ^n_3 = \cM  + \cL,
\end{equation}
where $\cL$ is given by \eqref{def_L}. We next present our main result.

\begin{proposition} \label{Q3_subset_Q}
The following relation holds:
\begin{equation} \label{Q3_in_Q}
\cQ^n_3 \subseteq \cQ^n,    
\end{equation}
where $\cQ^n_3$ and $\cQ^n$ are defined as in \eqref{def_Q3} and \eqref{def_Q}, respectively.
\end{proposition}
\begin{proof}
Let $Q \in \cQ^n_3$. Then, there exist a perfect graph $G = (V_G,E_G) \in \mathfrak{G}$ and $w \in \R^n_{++}$ such that $Q = \widehat{Q} + \lambda E$ for some $\widehat{Q} \in {\cal M}(G,w)$ and $\lambda \in \R$. Since $G$ is a perfect graph and $\widehat{Q} \in {\cal M}(G,w)$, it follows from Theorem~\ref{mwss_stqp}, the relation \eqref{lov_sch_per}, and Proposition~\ref{dnn_char_M} that $\ell(\widehat{Q}) = \nu(\widehat{Q})$, i.e., $\widehat{Q} \in \cQ^n$. The inclusion \eqref{Q3_in_Q} directly follows from Lemma~\ref{shift_invariance}.
\end{proof}
 
We close this section by noting that the membership problem in $\cQ^n_3$ can, in theory, be solved in polynomial time. Given $Q \in \cS^n$, let $G = (V_G,E_G)$, where $V_G = \{1,\ldots,n\}$ and 
\begin{equation} \label{def_EG_by_Q}
E_G = \left\{(i,j): 1 \leq i < j \leq n, \quad 2 Q_{ij} < Q_{ii} + Q_{jj} \right\}.
\end{equation}
There are two cases. If $E_G = \emptyset$, then let $\gamma = \min\limits_{1 \leq i \leq j \leq n} Q_{ij} - 1$ and define $\widehat{Q} = Q - \gamma E \in \cN^n$. Then, $\widehat{Q}$ has strictly positive entries and $2 \widehat{Q}_{ij} \geq \widehat{Q}_{ii} + \widehat{Q}_{jj}$ for each $1 \leq i < j \leq n$ by \eqref{def_EG_by_Q}. By defining $w \in \R^n_{++}$ with $w_k = 1/\widehat{Q}_{kk} > 0,~ k = 1,\ldots,n$, it follows that $\widehat{Q} \in {\cal M}(G,w)$ and $G$ is clearly perfect since it contains no edges. Therefore, $Q = \widehat{Q} + \gamma E \in \cQ^n_3$. It is worth noticing that any such matrix $Q$ also belongs to $\cQ^n_1$, where $\cQ^n_1$ is given by \eqref{def_cQ1}.

Suppose, on the other hand, that $E_G \neq \emptyset$. We first observe that, by \eqref{def_Q3}, a necessary condition for $Q \in \cQ^n_3$ is given by $Q_{ij} = \alpha$ for each $(i,j) \in E_G$, where $\alpha \in \R$. Therefore, let $\kappa_1 = \min\limits_{(i,j) \in E_G} Q_{ij}$ and $\kappa_2 = \max\limits_{(i,j) \in E_G} Q_{ij}$. If $\kappa_1 < \kappa_2$, then $Q \not \in \cQ^n_3$ by the previous necessary condition. Otherwise, let $\kappa = \kappa_1 = \kappa_2$ and $\widehat{Q} = Q - \kappa E$, Note that $\widehat{Q}_{ij} = 0$ for each $(i,j) \in E_G$ and $\widehat{Q}_{ij} \ge \widehat{Q}_{ii} + \widehat{Q}_{jj}$ for each $(i,j) \in E_{\overbar{G}}$. If $\widehat{Q}$ has strictly positive diagonal entries, then we can ensure that $\widehat{Q} \in {\cal M}(G,w)$ by similarly defining $w \in \R^n_{++}$ with $w_k = 1/Q^s_{kk} > 0,~ k = 1,\ldots,n$. Then, one can check in polynomial time if $G = (V_G,E_G)$ is a perfect graph~\cite{CCLSV05} and accordingly decide if $Q \in \cQ^n_3$. Finally, if $\widehat{Q} = Q - \kappa E$ does not have strictly positive diagonal entries, then $Q \not \in \cQ^n_3$. In the latter case, note, however, that $Q \in \cQ^n_1$, and therefore $Q \in \cQ^n$ by Proposition~\ref{Q1subsetQ}.

Conversely, for any perfect graph $G = (V_G,E_G)$ and any $w \in \R^n_{++}$, choosing any matrix $\widehat{Q} \in {\cal M}(G,w)$ and any $\lambda \in \R$, and defining $Q = \widehat{Q} + \lambda E$, we ensure that $Q \in \cQ^n$ by Proposition~\ref{Q3_subset_Q}.

\subsection{Relations Among Three Families}

In Sections~\ref{min_diag}, \ref{psd_on_e_perp}, and \ref{max_w_clique}, we have explicitly identified three families of instances of (StQP) that admit exact doubly nonnegative relaxations. In this section, we present numerical examples illustrating that neither of these subsets is contained in any of the other two subsets. We also present an example that shows the existence of an instance that belongs to $\cQ^n$ but is not contained in any of the three families.

\begin{example} \label{ex1}
Let
\[
Q  = \begin{bmatrix} 0 & 1 & 3 & 2 & 0\\
1 & 3 & 1 & 3 & 2 \\
3 & 1 & 2 & 1 & 3 \\
2 & 3 & 1 & 1 & 0 \\
0 & 2 & 3 & 0 & 1
\end{bmatrix}.
\]
Observe that $\min\limits_{1 \leq i \leq j \leq n}Q_{ij} = Q_{11} = 0$, which implies that $Q \in \cQ^5_1$ by \eqref{def_cQ1} and $Q \in \cQ^5$ by Proposition~\ref{Q1subsetQ}. Indeed, we have $\Omega(Q) = \{e_1\}$ and $\nu(Q) = \ell(Q) =  0$. 

Let $d = [4,-1,-1,-1,-1]^T \in \R^5$. Note that $e^T d = 0$. However, $d^T Q d = -21 < 0$, which implies that $Q \not \in \cQ^5_2$ by \eqref{def_cQ2}. 

Finally, by \eqref{def_EG_by_Q}, 
\[
E_G = \left\{(1,2),(1,5),(2,3),(3,4),(4,5)\right\}.
\]
Since $E_G \neq \emptyset$, we have $\kappa_1 = \min\limits_{(i,j) \in E_G} Q_{ij} = Q_{15} = 0 < \kappa_2 = \max\limits_{(i,j) \in E_G} Q_{ij} = Q_{12} = 1$, which implies that $Q \not \in \cQ^5_3$ by the discussion at the end of Section~\ref{max_w_clique}. It follows that $Q \in \cQ^5_1 \backslash \left(\cQ^5_2 \cup \cQ^5_3\right)$.
\end{example} 

\begin{example} \label{ex2}
Let
\[
Q = \begin{bmatrix} 2 & 0 & 0 & 0 & 0\\
0 & 2 & 1 & 0 & 0 \\
0 & 1 & 1 & 0 & 0 \\
0 & 0 & 0 & 1 & 1 \\
0 & 0 & 0 & 1 & 1
\end{bmatrix}.
\]
Note that $Q \in {\cal PSD}^5$, which implies that $Q \in \cQ^5_2$ by \eqref{def_cQ2} and $Q \in \cQ^5$ by Proposition~\ref{Q2subsetQ}. An optimal solution is given by $x^* = [0.2,0,0.4,0.4,0]^T$ and $\ell(Q) = \nu(Q) = 0.4$. 

Observe that $\min\limits_{1 \leq i \leq j \leq n}Q_{ij} = Q_{12} = 0 < \min\limits_{k = 1,\ldots,n} Q_{kk} = Q_{22} = 1$, which implies that $Q \not \in \cQ^5_1$ by \eqref{def_cQ1}. 

Finally, by \eqref{def_EG_by_Q}, 
\[
E_G = \left\{(1,2),(1,3),(1,4),(1,5),(2,3),(2,4),(2,5),(3,4),(3,5) \right\}.
\]
Since $E_G \neq \emptyset$, we have $\kappa_1 = \min\limits_{(i,j) \in E_G} Q_{ij} = Q_{12} = 0 < \kappa_2 = \max\limits_{(i,j) \in E_G} Q_{ij} = Q_{23} = 1$, which implies that $Q \not \in \cQ^5_3$ by the discussion at the end of Section~\ref{max_w_clique}. It follows that $Q \in \cQ^5_2 \backslash \left(\cQ^5_1 \cup \cQ^5_3\right)$.
\end{example} 

\begin{example} \label{ex3}
Let
\[
Q = \begin{bmatrix} 
1 & 1 & 1 & 1 & 1\\
1 & 1 & 1 & 1 & 1 \\
1 & 1 & 1 & 1 & 1 \\
1 & 1 & 1 & 1 & 0 \\
1 & 1 & 1 & 0 & 1
\end{bmatrix}.
\]
By \eqref{def_EG_by_Q}, 
\[
E_G = \left\{(4,5) \right\}.
\]
Since $E_G \neq \emptyset$, we have $\kappa_1 = \min\limits_{(i,j) \in E_G} Q_{ij} = Q_{45} = 0 = \kappa_2 = \max\limits_{(i,j) \in E_G} Q_{ij} = Q_{45} = 0 = \kappa$. Following the discussion at the end of Section~\ref{max_w_clique}, let $\widehat{Q} = Q - \kappa E = Q$. Clearly, the diagonal entries of $Q$ are all equal to 1. Therefore, we define $w = e \in \R^5$. It follows that $Q \in \cM(G,w)$, where 
$G = (V_G,E_G)$ and $V_G = \{1,2,3,4,5\}$. Note that neither $G$ nor $\overbar{G}$ contains an odd cycle of length at least five as an induced subgraph. It follows that $G$ is a perfect graph, which implies that $Q \in \cQ^5_3$ by \eqref{def_Q3}. By Theorem~\ref{mwss_stqp} and Proposition~\ref{dnn_char_M}, we obtain $\ell(Q) = \nu(Q) = 1/2$ and the unique optimal solution is given by $x^* = [0,0,0,0.5,0.5]^T$.

Observe that $\min\limits_{1 \leq i \leq j \leq n}Q_{ij} = Q_{45} = 0 < \min\limits_{k = 1,\ldots,n} Q_{kk} = Q_{11} = 1$, which implies that $Q \not \in \cQ^5_1$ by \eqref{def_cQ1}. 

Finally, let $d = [4,-1,-1,-1,-1]^T \in \R^5$. Note that $e^T d = 0$. However, $d^T Q d = -2 < 0$, which implies that $Q \not \in \cQ^5_2$ by \eqref{def_cQ2}. It follows that $Q \in \cQ^5_3 \backslash \left(\cQ^5_1 \cup \cQ^5_2\right)$.
\end{example}

As illustrated by Examples~\ref{ex1}, \ref{ex2}, and \ref{ex3}, each of the three sets $\cQ_1$, $\cQ_2$, and $\cQ_3$ may contain an element that does not belong to the other two. The final example in this section illustrates that there exist matrices for which the corresponding (StQP) instance admits an exact doubly nonnegative relaxation but do not belong to any of the three sets $\cQ^n_1$, $\cQ^n_2$, and $\cQ^n_3$.

\begin{example} \label{ex4}
Let
\[
Q = \begin{bmatrix} 
2 & 2 & 2 & 2 & 2\\
2 & 2 & 2 & 2 & 2 \\
2 & 2 & 2 & 1 & 2 \\
2 & 2 & 1 & 2 & 0 \\
2 & 2 & 2 & 0 & 2
\end{bmatrix}.
\]
Note that $\min\limits_{1 \leq i \leq j \leq n}Q_{ij} = Q_{45} = 0 < \min\limits_{k = 1,\ldots,n} Q_{kk} = Q_{11} = 2$, which implies that $Q \not \in \cQ^5_1$ by \eqref{def_cQ1}.

Let $d = [4,-1,-1,-1,-1]^T \in \R^5$. Note that $e^T d = 0$. However, $d^T Q d = -6 < 0$, which implies that $Q \not \in \cQ^5_2$ by \eqref{def_cQ2}. 

By \eqref{def_EG_by_Q}, 
\[
E_G = \left\{(3,4),(4,5) \right\}.
\]
Since $E_G \neq \emptyset$, we have $\kappa_1 = \min\limits_{(i,j) \in E_G} Q_{ij} = Q_{45} = 0 < \kappa_2 = \max\limits_{(i,j) \in E_G} Q_{ij} = Q_{34} = 1$, which implies that $Q \not \in \cQ^5_3$ by the discussion at the end of Section~\ref{max_w_clique}. It follows that $Q \not \in  \left(\cQ^5_1 \cup \cQ^5_2 \cup \cQ^5_3\right)$.

On the other hand, an optimal solution of the corresponding instance of (StQP) is given by $x^* = [0,0,0,0.5,0.5]^T$, and $\nu(Q) = 1$. Finally,
\[
Q - \left( (x^*)^T Q x^* \right) E = \begin{bmatrix} 
1 & 1 & 1 & 1 & 1\\
1 & 1 & 1 & 1 & 1 \\
1 & 1 & 1 & 0 & 1 \\
1 & 1 & 0 & 1 & -1 \\
1 & 1 & 1 & -1 & 1
\end{bmatrix}
= \underbrace{\begin{bmatrix} 1 & 0 & 0 & 0 & 0 \\
0 & 1 & 0 & 0 & 0 \\
0 & 0 & 1 & 0 & 0 \\
0 & 0 & 0 & 1 & -1 \\
0 & 0 & 0 & -1 & 1
\end{bmatrix}}_{P}
+
\underbrace{\begin{bmatrix} 0 & 1 & 1 & 1 & 1 \\
1 & 0 & 1 & 1 & 1 \\
1 & 1 & 0 & 0 & 1 \\
1 & 1 & 0 & 0 & 0 \\
1 & 1 & 1 & 0 & 0
\end{bmatrix}}_{N},
\]
which implies that $Q - \left( (x^*)^T Q x^* \right) E \in {\cal SPN}^5$ since $P \in {\cal PSD}^5$ and $N \in {\cal N}^5$. Therefore, by Proposition~\ref{charac_1}, it follows that $Q \in \cQ_{x^*}$, i.e., $\ell(Q) = \nu(Q) = 1$, which implies that $Q \in \cQ^5$. We conclude that $Q \in \cQ^5 \backslash \left(\cQ^5_1 \cup \cQ^5_2 \cup \cQ^5_3\right)$.
\end{example}

\section{Relations with Maximal Cliques of the Convexity Graph} \label{max_clique_conv_graph}

In this section, we establish several relations between the tightness of the doubly nonnegative relaxation of an instance of (StQP) and the maximal cliques of the so-called convexity graph associated with the matrix $Q \in \cS^n$.

\subsection{Convexity Graph}

For a given $Q \in \cS^n$, one can define a simple undirected graph $G_Q = (V_Q,E_Q)$, referred to as the \emph{convexity graph} of $Q$, where the set of vertices is given by $V_Q = \{1,\ldots,n\}$ and the set of edges is defined as 
\begin{equation} \label{edge_def}
E_Q =  \{(i,j): 2Q_{ij} <  Q_{ii} + Q_{jj} , \quad 1 \le i < j \le n\}. 
\end{equation}

It is easy to verify the following shift invariance property of the convexity graph.
\begin{equation} \label{conv_graph_shift}
G_{Q + \lambda E} = G_{Q}, \quad \forall~ Q \in \cS^n, \quad \forall~\lambda \in \R. 
\end{equation}

Recall that a clique in a simple undirected graph is a set of mutually adjacent vertices. For a given $Q \in \cS^n$, the next result provides a useful connection between the cliques of the convexity graph $G_Q$ and the index set $A(x)$ of an optimal solution $x \in \Omega(Q)$ of the corresponding instance of (StQP), where $A(\cdot)$ is given by \eqref{def_P}.

\begin{theorem}[Scozzari and Tardella, 2008] \label{Scozzari}
Given $Q \in \cS^n$, there exists an optimal solution $x^* \in \Omega(Q)$ of the corresponding instance of (StQP) such that the vertices corresponding to $A(x^*)$ form a clique in the convexity graph $G_Q = (V_Q,E_Q)$, where $V_Q = \{1,\ldots,n\}$, and $A(\cdot)$ and $E_Q$ are given by \eqref{def_P} and \eqref{edge_def}, respectively.
\end{theorem}

Let $w \in \R^n_{++}$ and let $G = (V_G,E_G)$ be a graph with $V_G = \{1,\ldots,n\}$. Note that, for any $Q \in \cM(G,w)$, where $\cM(G,w)$ is given by \eqref{def_M_class}, the convexity graph of $Q$ is given by $G_Q = G$. Therefore, by Theorem~\ref{Scozzari}, the corresponding (StQP) instance has an optimal solution $x^* \in \Omega(Q)$ such that $A(x^*)$ induces a clique in $G$. Indeed, for any maximum weight clique $C \subseteq V_Q $, an optimal solution of the corresponding (StQP) presented in  Theorem~\ref{mwss_stqp} is given by (see, e.g.,~\cite{ref:Gibbons1997})
\[
x^*_j = \begin{cases} \frac{w_j}{w(C)}, & \textrm{if}~j \in C, \\
0 & \textrm{otherwise,}
\end{cases} 
\]
where $w(C) = \sum\limits_{j \in C} w_j$. Note that $A(x^*) = C$, which is a clique in $G_Q$.

\subsection{Maximal Cliques of the Convexity Graph}
 
For a given simple undirected graph $G = (V,E)$, a clique $C \subseteq V$ is said to be \emph{maximal} if it is not a proper subset of a larger clique in $G$. For a given $Q \in \cS^n$, the following lemma establishes useful relations between $\ell(Q)$, $\nu(Q)$, and the maximal cliques of the convexity graph. 

\begin{lemma} \label{max_cliques}
For a given $Q \in \cS^n$, let $G_Q = (V_Q,E_Q)$ denote the convexity graph of $Q$ and let $\mathfrak{C}$ denote the collection of all maximal cliques of $G_Q$. Then, 
\begin{equation} \label{max_clique_rels}
\ell(Q) \leq \min\limits_{C \in \mathfrak{C}} \ell(Q_{CC}) \leq \min\limits_{C \in \mathfrak{C}} \nu(Q_{CC}) = \nu(Q).
\end{equation}
Furthermore, each of the two inequalities is satisfied as an equality if and only if $Q \in \cQ^n$, where $\cQ^n$ is given by \eqref{def_Q}.
\end{lemma}
\begin{proof}
Consider the first inequality in~\eqref{max_clique_rels}. For any maximal clique $C \in \mathfrak{C}$, consider any optimal solution $X^* \in {\cal DN}^{|C|}$ of (DN-P) corresponding to $Q_{CC} \in \cS^{|C|}$. By Lemma~\ref{lem_gen_rels} (ii) and (v), $X^*$ can be extended to a solution $\widehat{X} \in \cS^n$ by defining $\widehat{X}_{CC} = X^*$ and $\widehat{X}_{ij} = 0$ if $i \not \in C$ or $j \not \in C$. It follows that $\widehat{X} \in {\cal DN}^n$ is a feasible solution of (DN-P) corresponding to $Q$ and
\[
\ell(Q) \leq \langle Q, \widehat{X} \rangle = \langle Q_{CC}, X^* \rangle = \ell(Q_{CC}),
\]
which establishes the first inequality in~\eqref{max_clique_rels}.

The second inequality in~\eqref{max_clique_rels} immediately follows from~\eqref{lb_dnn}.

Consider now the last equality in~\eqref{max_clique_rels}. For any maximal clique $C \in \mathfrak{C}$, we have
\[
\nu(Q_{CC}) = \min\limits_{w \in \Delta_{|C|}} w^T Q_{CC} \, w = \min\limits_{x \in \Delta_n}\left\{x^T Q x: x_j = 0,~j \not \in C \right\} \geq \nu(Q),
\]
which implies that $\nu(Q) \leq \min\limits_{C \in \mathfrak{C}} \nu(Q_{CC})$.

By Theorem~\ref{Scozzari}, there exists an $x^* \in \Omega(Q)$ such that the subgraph of $G_Q$ induced by $A(x^*)$ is a clique. Let $\widehat{C} \in \mathfrak{C}$ denote any maximal clique of $G(Q)$ such that $A(x^*) \subseteq \widehat{C}$. Since $e^T x^* = e_{\widehat{C}}^T x^*_{\widehat{C}} = 1$, it follows that
\[
 \nu(Q) = (x^*)^T Q x^* = \left(x^*_{\widehat{C}} \right)^T Q_{\widehat{C}\widehat{C}} \left(x^*_{\widehat{C}} \right) \geq \nu(Q_{\widehat{C}\widehat{C}}) \geq \min\limits_{C \in \mathfrak{C}} \nu(Q_{CC}),
\]
which establishes the reverse inequality. Therefore, $\nu(Q) = \min\limits_{C \in \mathfrak{C}} \nu(Q_{CC})$.

The last assertion immediately follows from \eqref{def_Q}.
\end{proof}

Next, we present two examples illustrating that each of the two inequalities in~\eqref{max_clique_rels} can be strict. 

\begin{example}\label{ex5}
Let
\[
Q = \begin{bmatrix} 
1 & 0 & 0.9 & 0.9 & 0 \\
0 & 1 & 0 & 0.9 & 0.9  \\
0.9 & 0 & 1 & 0 & 0.9  \\
0.9 & 0.9 & 0 & 1 & 0  \\
0 & 0.9 & 0.9 & 0 & 1 
\end{bmatrix}.
\]
The convexity graph $G_Q = (V_Q,E_Q)$ is given by
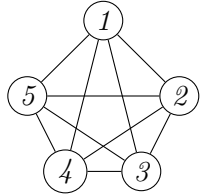
\begin{center}
\begin{tikzpicture}
\node[vertex] (4) at (0,0) {4};
\node[vertex] (3) at (1,0) {3};
\node[vertex] (2) at (1.5,1) {2};
\node[vertex] (1) at (0.5,2) {1};
\node[vertex] (5) at (-0.5,1) {5};
\path[color=black]
(1) edge (2)
(2) edge (3)
(3) edge (4)
(4) edge (5)
(1) edge (3)
(5) edge (1)
(2) edge (4)
(3) edge (5)
(2) edge (5)
(1) edge (4);
\end{tikzpicture}
\end{center}
Since $G_Q$ is a complete graph, the only maximal clique in $G_Q$ is $C_1 = \{1,2,3,4,5\}$, i.e., $\mathfrak{C} = \{C_1\}$. In this example, 
\[
0.4472 \approx \ell(Q) = \min\limits_{C \in \mathfrak{C}} \ell(Q_{CC}) = \ell(Q_{C_1 C_1}) < \min\limits_{C \in \mathfrak{C}} \nu(Q_{CC}) = \nu(Q_{C_1 C_1}) = \nu(Q) \approx 0.4872, 
\]
which implies that the first inequality in \eqref{max_clique_rels} is satisfied with equality, whereas the second inequality is strict.
\end{example}

\begin{example}
 \label{ex6}
Let
\[
Q = \begin{bmatrix} 
1 & 0 & 0.9 & 1 & 0\\
0 & 1 & 0 & 1 & 1 \\
0.9 & 0 & 1 & 0 & 1 \\
1 & 1 & 0 & 1 & 0 \\
0 & 1 & 1 & 0 & 1
\end{bmatrix}.
\]
The convexity graph $G_Q = (V_Q,E_Q)$ is given by
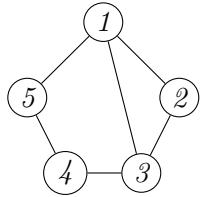
\begin{center}
\begin{tikzpicture}
\node[vertex] (4) at (0,0) {4};
\node[vertex] (3) at (1,0) {3};
\node[vertex] (2) at (1.5,1) {2};
\node[vertex] (1) at (0.5,2) {1};
\node[vertex] (5) at (-0.5,1) {5};
\path[color=black]
(1) edge (2)
(2) edge (3)
(3) edge (4)
(4) edge (5)
(1) edge (3)
(5) edge (1);
\end{tikzpicture}
\end{center}
Therefore, $\mathfrak{C} = \{C_1,C_2,C_3,C_4,C_5\}$, where 
\[
C_1 = \{1,2,3\}, \quad C_2 = \{1,5\}, \quad C_3 =\{3,4\}, \quad C_4 = \{4,5\}.
\]
In this example, 
\[
0.4472 \approx \ell(Q) < \min\limits_{C \in \mathfrak{C}} \ell(Q_{CC}) = \ell(Q_{C_1 C_1}) = \min\limits_{C \in \mathfrak{C}} \nu(Q_{CC}) = \nu(Q_{C_1 C_1}) = \nu(Q) \approx 0.4872, 
\]
which implies that the first inequality in \eqref{max_clique_rels} is strict, whereas the second inequality is satisfied with equality.
\end{example}

By Lemma~\ref{max_cliques}, unless $G_Q$ is a complete graph, an instance of (StQP) can be decomposed into smaller instances of (StQP) each of which corresponds to a maximal clique of the convexity graph $G_Q$ (see Example~\ref{ex6}). In particular, if $G_Q$ has several connected components, then the problem naturally decomposes into subproblems corresponding to each connected component. Furthermore, the lower bound $\ell(Q)$ can be improved if one focuses instead on the principal submatrices of $Q$ corresponding to maximal cliques of the convexity graph $G_Q$ (see Example~\ref{ex6}). We remark that, in the worst case, a graph with $n$ vertices may have as many as $3^{n/3}$ of maximal cliques~\cite{moon1965cliques}. On the other hand, several classes of graphs including planar and chordal graphs have a polynomial number of maximal cliques (see, e.g.,~\cite{RosgenS07} and the references therein). Therefore, on instances of (StQP) with such a convexity graph, Lemma~\ref{max_cliques} implies that the original problem can be decomposed into a polynomial number of smaller problems and that the lower bound can potentially be improved by focusing only on the doubly nonnegative relaxations of the smaller problems corresponding to maximal cliques of $G_Q$.  

The next result characterizes the set of instances of (StQP) for which the second inequality in \eqref{max_clique_rels} is satisfied with equality. 

\begin{lemma} \label{min_cliq_DN_exact}
For a given $Q \in \cS^n$, let $G_Q = (V_Q,E_Q)$ denote the convexity graph of $Q$ and let $\mathfrak{C}$ denote the collection of all maximal cliques of $G_Q$. Then, 
\begin{equation} \label{dnn_sub_exact}
\min\limits_{C \in \mathfrak{C}} \ell(Q_{CC}) = \nu(Q)
\end{equation}
if and only if there exists $C^* \in \mathfrak{C}$ such that
\begin{equation} \label{dnn_sub_exact_2}
Q_{C^* C^*} \in \cQ^{|C^*|}, \quad \textrm{and} \quad \min\limits_{C \in \mathfrak{C}} \ell(Q_{CC}) = \ell(Q_{C^* C^*}),
\end{equation}
where $\cQ^n$ is given by \eqref{def_Q}.
\end{lemma}
\begin{proof}
Suppose that \eqref{dnn_sub_exact} holds. By Lemma~\ref{max_cliques},
\[
\min\limits_{C \in \mathfrak{C}} \ell(Q_{CC}) = \min\limits_{C \in \mathfrak{C}} \nu(Q_{CC}) = \nu(Q). 
\]
Then, by \eqref{lb_dnn}, there exists $C^* \in \mathfrak{C}$ such that $\min\limits_{C \in \mathfrak{C}} \ell(Q_{CC}) = \ell(Q_{C^* C^*}) = \nu(Q_{C^* C^*}) = \nu(Q)$, which implies that \eqref{dnn_sub_exact_2} is satisfied.

Conversely, if \eqref{dnn_sub_exact_2} holds, it follows from Lemma~\ref{max_cliques} that
\[
\min\limits_{C \in \mathfrak{C}} \ell(Q_{CC}) = \ell(Q_{C^* C^*}) = \nu(Q_{C^* C^*}) \leq \min\limits_{C \in \mathfrak{C}} \nu(Q_{CC}) = \nu(Q) \leq \nu(Q_{C^* C^*}), 
\]
which implies that \eqref{dnn_sub_exact} holds.
\end{proof}

For instance, for each $C \in \mathfrak{C}$ in Example~\ref{ex6}, since $|C| \leq 4$, we have $Q_{CC} \in \cQ^{|C|}$ by \eqref{diananda}, which implies the existence of a maximal clique (i.e., $C_1$) that satisfies \eqref{dnn_sub_exact_2}, and therefore \eqref{dnn_sub_exact} by Lemma~\ref{min_cliq_DN_exact}. On the other hand, since there is only one maximal clique $C_1$ in Example~\ref{ex5} and $Q_{C_1} = Q \not \in \cQ^5$, Lemma~\ref{min_cliq_DN_exact} implies that the second inequality in \eqref{max_clique_rels} is strict.

\subsection{Matrix Completion and SPN Completable Graphs}

Lemma~\ref{min_cliq_DN_exact} establishes the equivalence of the conditions \eqref{dnn_sub_exact} and \eqref{dnn_sub_exact_2}. However, as illustrated by Example~\ref{ex6}, neither of these conditions implies that $Q \in \cQ^n$, where $\cQ^n$ is given by \eqref{def_Q}.

In this section, we identify an additional condition under which either of the conditions \eqref{dnn_sub_exact} and \eqref{dnn_sub_exact_2} implies that $Q \in \cQ^n$. 

First, we define the matrix completion problem. We mostly follow the discussion in \cite{ref:shaked2016}. Let $V = \{1,\ldots,n\}$ and let $F \subseteq V \times V$ be a set with the following properties.
\begin{equation} \label{prop_set_F}
(i,i) \in F, \quad \forall~i \in V, \quad \textrm{and} \quad (i,j) \in F \Longleftrightarrow (j,i) \in F, \quad 1 \leq i < j \leq n.    
\end{equation}

For a given set $F \subseteq V \times V$ that satisfies \eqref{prop_set_F}, a partial matrix $B \in \cS^n$ is a matrix whose entries $B_{ij}$ are specified if and only if $(i,j) \in F$. For a given set ${\cal K} \subseteq \cS^n$, the matrix completion problem is concerned with finding a matrix $\widehat{B} \in {\cal K}$ such that
\begin{equation} \label{matrix_completion}
\widehat{B}_{ij} = B_{ij}, \forall~(i,j) \in F, \quad \textrm{and} \quad \widehat{B} \in {\cal K}.
\end{equation}
The SPN completion problem is concerned with whether a partial matrix $B \in \cS^n$ is SPN completable, i.e., whether there exists $\widehat{B} \in {\cal K}$ that satisfies \eqref{matrix_completion}, with ${\cal K} = {\cal SPN}^n$. By Lemma~\ref{lem_gen_rels} (iv), if a partial matrix $B$ is SPN completable, then each of its fully specified $r \times r$ principal submatrices should belong to ${\cal SPN}^r$, where $r = 1,\ldots,n$. A partial matrix $B$ that satisfies this necessary condition is called a partial SPN matrix. 

For a given simple undirected graph $G = (V_G,E_G)$, where $V_G = \{1,\ldots,n\}$, one can associate an SPN completion problem, where $B_{ij} = B_{ji}$ is specified if and only if $(i,j) \in E_G$ or $i = j$. Such a matrix $B$ is called a $G$-partial matrix. A graph $G$ is said to be SPN completable if every $G$-partial SPN matrix is SPN completable. 

The following result in~\cite{ref:shaked2016} presents a full characterization of SPN completable graphs.

\begin{theorem}[Shaked-Monderer et al., 2016] \label{spn_comp} 
Let $G = (V_G,E_G)$ be a graph, where $V_G = \{1,\ldots,n\}$. $G$ is SPN completable if and only if every odd cycle in $G$ induces a complete subgraph of $G$.
\end{theorem}

Recall that each of the two inequalities in \eqref{max_clique_rels} is satisfied with equality if and only if $Q \in \cQ^n$ by Lemma~\ref{max_cliques}. Furthermore, Lemma~\ref{min_cliq_DN_exact} gives a full characterization of instances for which the latter inequality in \eqref{max_clique_rels} is satisfied with equality. Example~\ref{ex6} illustrates that there exists a matrix $Q \in \cS^5$ that satisfies the conditions of Lemma~\ref{min_cliq_DN_exact} but $Q \not \in \cQ^5$. In the next result, for a given $Q \in \cS^n$, under the additional assumption that the convexity graph $G_Q = (V_Q,E_Q)$ is SPN completable, we show that the exactness of the second inequality in \eqref{max_clique_rels} implies the exactness of the first inequality, thereby establishing $Q \in \cQ^n$.

\begin{proposition} \label{SPNcomp_nec_suf}
Let $Q \in \cS^n$ be a matrix such that its convexity graph $G_Q = (V_Q,E_Q)$ is SPN completable. Then, $Q \in \cQ^n$, where $\cQ^n$ is given by \eqref{def_Q}, if and only if \eqref{dnn_sub_exact} is satisfied.
\end{proposition} 
\begin{proof}
Let $Q \in \cS^n$ be a matrix such that its convexity graph $G_Q = (V_Q,E_Q)$ is SPN completable. By Lemma~\ref{max_cliques}, if $Q \in \cQ^n$, then the condition \eqref{dnn_sub_exact} is satisfied.

Conversely, suppose that \eqref{dnn_sub_exact} is satisfied. Then, 
\[
\min\limits_{C \in \mathfrak{C}} \ell(Q_{CC}) = \nu(Q),
\]
where $\mathfrak{C}$ is the collection of all maximal cliques of $G_Q$. For each $C \in \mathfrak{C}$, we have $Q_{CC} - \ell(Q_{CC}) E \in {\cal SPN}^{|C|}$ by (DN-D) corresponding to $Q_{CC}$. Since $\ell(Q_{CC}) \geq \nu(Q)$, it follows that 
\begin{equation} \label{max_cliques_spn}
Q_{CC} - \nu(Q) E = Q_{CC} - \ell(Q_{CC}) E + (\ell(Q_{CC}) - \nu(Q))E \in {\cal SPN}^{|C|}, \quad \forall~C \in \mathfrak{C}.
\end{equation}
Consider the following $G_Q$-partial matrix $B \in \cS^n$:
\[
B_{ij} = \begin{cases}
Q_{ij} - \nu(Q), & \textrm{if} ~(i,j) \in E_Q, \\
Q_{ii} - \nu(Q), & i = 1,\ldots,n.
\end{cases}
\]
Note that every fully specified principal submatrix of $B$ corresponds to a maximal clique $C \in \mathfrak{C}$. By \eqref{max_cliques_spn}, it follows that $B$ is a $G_Q$-partial SPN matrix. Since $G_Q$ is an SPN completable graph by the hypothesis, there exists a matrix $\widehat{B} \in \cS^n$ such that $\widehat{B} \in {\cal SPN}^n$ and 
\begin{equation} \label{A_hat_rels}
\widehat{B}_{ij} = B_{ij}, \quad (i,j) \in E_Q, \quad \widehat{B}_{ii} = B_{ii}, \quad i = 1,\ldots,n.     
\end{equation}
Therefore, there exist $\widehat{P} \in {\cal PSD}^n$ and $\widehat{N} \in \cN^n$ such that $\widehat{B} = \widehat{P} + \widehat{N}$. Without loss of generality, we may assume that $\widehat{N}_{ii} = 0$ for each $i = 1,\ldots,n$ by simply increasing all the diagonal elements of $\widehat{P}$ accordingly if necessary. Therefore, 
\begin{eqnarray} \label{P_A_rels}
\widehat{P}_{ii} & = & \widehat{B}_{ii} - \widehat{N}_{ii} =  \widehat{B}_{ii} =  Q_{ii} - \nu(Q), \quad i = 1,\ldots,n, \label{P_A_rels_1}\\
\widehat{P}_{ij} & = & \widehat{B}_{ij} - \widehat{N}_{ij} \leq \widehat{B}_{ij} =  Q_{ij} - \nu(Q), \quad (i,j) \in E_Q. \label{P_A_rels_2}
\end{eqnarray}
Let us fix $(i,j)$ such that $i \neq j$ and $(i,j) \not \in E_Q$. By the definition of $E_Q$ in \eqref{edge_def}, we have $2 Q_{ij} \geq Q_{ii} + Q_{ij}$. Therefore, for any such $(i,j)$, since $\widehat{P} \in {\cal PSD}^n$,
\begin{equation} \label{P_A_rels_3}
\widehat{P}_{ij} \leq \frac{1}{2} \left( \widehat{P}_{ii} + \widehat{P}_{jj} \right) = \frac{1}{2} \left(Q_{ii} + Q_{jj}\right) - \nu(Q) \leq Q_{ij} - \nu(Q) 
\end{equation}
by \eqref{P_A_rels_1}. Finally, combining \eqref{P_A_rels_1}, \eqref{P_A_rels_2}, and \eqref{P_A_rels_3}, we conclude that there exists a matrix $N \in \cN^n$ such that $N_{ii} = 0$ for each $i = 1,\ldots,n$ and 
\[
Q - \nu(Q) E = \widehat{P} + N,
\]
which implies that $Q - \nu(Q) E \in {\cal SPN}^n$. Therefore, by (DN-D), we obtain $\ell(Q) \geq \nu(Q)$, which, together with \eqref{lb_dnn}, implies that $\ell(Q) = \nu(Q)$, or equivalently, that $Q \in \cQ^n$. 
\end{proof}

\begin{example} \label{ex7}
Consider Example~\ref{ex4}. Note that $Q \in \cQ^5 \backslash \left(\cQ^5_1 \cup \cQ^5_2 \cup \cQ^5_3\right)$. The convexity graph $G_Q = (V_Q,E_Q)$ is given by $V_Q = \{1,2,3,4,5\}$ and 
\[
E_Q =  \left\{(3,4),(4,5) \right\}.
\]
Clearly, $G_Q$ is SPN completable since it does not contain any odd cycle. The set of maximal cliques of $G_Q$ is given by $\mathfrak{C} = \{C_1,C_2\}$, where $C_1 = \{3,4\}$ and $C_2 = \{4,5\}$. Since $|C_1| = |C_2| = 2 \leq 4$, it follows that $\ell(Q_{C_1 C_1}) = \nu(Q_{C_1 C_1}) = 1.5$ and $\ell(Q_{C_2 C_2}) = \nu(Q_{C_2 C_2}) = 1$. Therefore, $\min\limits_{C \in \mathfrak{C}} \ell(Q_{CC}) = \min\limits_{C \in \mathfrak{C}} \nu(Q_{CC}) = \nu(Q) = 1$. Since $G_Q$ is SPN completable, it follows from Proposition~\ref{SPNcomp_nec_suf} that $Q \in \cQ^5$.
\end{example}

As illustrated by Example~\ref{ex7}, Proposition~\ref{SPNcomp_nec_suf} may be helpful for identifying a matrix $Q \in \cQ^n \backslash \left(\cQ^n_1 \cup \cQ^n_2 \cup \cQ^n_3\right)$. An interesting question is whether every matrix $Q \in \cQ^n \backslash \left(\cQ^n_1 \cup \cQ^n_2 \cup \cQ^n_3\right)$ satisfies the conditions of Proposition~\ref{SPNcomp_nec_suf}. We close this section by the following counterexample.

\begin{example} \label{ex8}
Let
\[
Q = \begin{bmatrix} 
2 & 0 & 0 & 2 & 1\\
0 & 2 & 0 & 2 & 2\\
0 & 0 & 2 & 0 & 2\\
2 & 2 & 0 & 2 & 0\\
1 & 2 & 2 & 0 & 2
\end{bmatrix}.
\]
The convexity graph $G_Q = (V_Q,E_Q)$ is given by
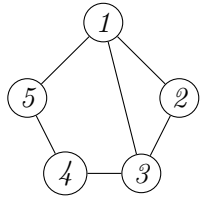
\begin{center}
\begin{tikzpicture}
\node[vertex] (4) at (0,0) {4};
\node[vertex] (3) at (1,0) {3};
\node[vertex] (2) at (1.5,1) {2};
\node[vertex] (1) at (0.5,2) {1};
\node[vertex] (5) at (-0.5,1) {5};
\path[color=black]
(1) edge (2)
(2) edge (3)
(3) edge (4)
(4) edge (5)
(1) edge (3)
(5) edge (1);
\end{tikzpicture}
\end{center}
Therefore, $\mathfrak{C} = \{C_1,C_2,C_3,C_4,C_5\}$, where 
\[
C_1 = \{1,2,3\}, \quad C_2 = \{1,5\}, \quad C_3 =\{3,4\}, \quad C_4 = \{4,5\}.
\]
In this example, an optimal solution is given by $x^* = [1/3,1/3,1/3,0,0]^T$ and $\nu(Q) = \nu(Q_{C_1 C_1}) = \ell(Q_{C_1 C_1}) = \min\limits_{C \in \mathfrak{C}} \ell(Q_{CC}) = 2/3$. One can numerically verify that $\ell(Q) = \nu(Q) = 2/3$, which implies that $Q \in \cQ^5$. 

On the other hand, it is easy to verify that $Q \not \in \cQ^5_1$. For $d = [-1, -1, -1, 2, 1]^T \in e^\perp$, we have $d^T Q d = -10 < 0$, which implies that $Q \not \in \cQ^5_2$. By a similar argument at the end of Section~\ref{max_w_clique}, we obtain $Q \not \in \cQ^5_3$. Finally, since the subgraph induced by the odd cycle consisting of the vertices $\{1,2,3,4,5\}$ is not a complete graph, $G_Q$  
is not SPN completable by Theorem~\ref{spn_comp}. It follows that $Q$ does not satisfy the conditions of Proposition~\ref{SPNcomp_nec_suf}.
\end{example}

\section{Standard Quadratic Programs with Positive Gaps} \label{pos_gap}

In this section, we focus on the set of instances of (StQP) for which there is a positive gap between the lower bound arising from the doubly nonnegative relaxation and the optimal value of (StQP), i.e., 
\begin{equation} \label{cQ_comp}
\cS^n \backslash \cQ^n = \left\{Q \in \cS^n: \ell(Q) < \nu(Q)\right\}. \end{equation}

We first present an algebraic characterization of such instances. Based on this characterization, we then propose a procedure for generating such an instance. 

\subsection{An Algebraic Characterization}

The next result gives a complete algebraic characterization of the set $\cS^n \backslash \cQ^n$.
    
\begin{proposition} \label{pos_gap_char}
Let $Q \in \cS^n$. Then $Q \in \cS^n \backslash \cQ^n$ if and only if there exist $\lambda \in \R$ and $M \in \partial ~{\cal COP}^n \backslash {\cal SPN}^n$ such that
\begin{equation} \label{inexact_decomp}
Q = \lambda E + M.    
\end{equation}
Furthermore, for any decomposition given by \eqref{inexact_decomp}, we have $\lambda = \nu(Q)$ and $\Omega(Q) = \mathbf{V}^M$, where $\mathbf{V}^M$ is given by \eqref{def_M_zeros}.
\end{proposition}
\begin{proof}
Let $Q \in \cS^n$ be such that $Q \in \cS^n \backslash \cQ^n$. Let $x^* \in \Omega(Q)$ be any optimal solution. Let us define $M = Q - \left( (x^*)^T Q x^* \right) E$. By Lemma~\ref{tech_res_cop},  $M \in \partial ~{\cal COP}^n$, where $\partial ~{\cal COP}^n$ is given by \eqref{cop_bd}. Since $Q \not \in \cQ^n$, it follows that $Q \not \in \cQ_{x^*}$, where $\cQ_{x^*}$ is given by \eqref{def_Qx}. By Proposition~\ref{charac_1}, $M \not \in {\cal SPN}^n$. It follows that $Q = \lambda E + M$, where $\lambda = (x^*)^T Q x^*$ and $M \in \partial ~{\cal COP}^n \backslash {\cal SPN}^n$.

Conversely, suppose that $Q = \lambda E + M$, where $\lambda \in \R$ and $M \in \partial ~{\cal COP}^n \backslash {\cal SPN}^n$. For any $x \in \Delta_n$, since $M \in {\cal COP}^n$,
\[
x^T Q x = \lambda + x^T M x \geq \lambda,
\]
which implies that $\nu(Q) \geq \lambda$. 
Since $M \in \partial ~{\cal COP}^n$, it follows that $\mathbf{V}^M \neq \emptyset$, where $\mathbf{V}^M$ is given by \eqref{def_M_zeros}. Then, for any $x^* \in \mathbf{V}^M$, we have $(x^*)^T Q x^* = \lambda + (x^*)^T M x^* = \lambda$, which implies that $\nu(Q) = \lambda$ and $x^* \in \Omega(Q)$. Suppose, for a contradiction, that $Q \in \cQ^n$. Then, $Q \in \cQ_{x^*}$ by \eqref{def_Qx}. By Proposition~\ref{charac_1}, $Q - ((x^*)^T Q x^*) E = Q - \nu(Q) E = Q - \lambda E = M \in {\cal SPN}^n$, which contradicts the hypothesis that $M \in \partial ~{\cal COP}^n \backslash {\cal SPN}^n$. Therefore, $Q \in \cS^n \backslash \cQ^n$.

For the last assertion, the argument in the previous paragraph already establishes that $\nu(Q) = \lambda$ and $\mathbf{V}^M \subseteq \Omega(Q)$. Conversely, since $M \in {\cal COP}^n$, we have $x^T Q x = \lambda + x^T M x > \lambda$ for any $x \in \Delta_n \backslash \mathbf{V}^M$, which implies that $\Omega(Q) \subseteq \mathbf{V}^M$, thereby establishing $\Omega(Q) = \mathbf{V}^M$.
\end{proof}

\subsection{Generating Standard Quadratic Programs with a Positive Gap}

Note that Proposition~\ref{pos_gap_char} presents a complete algebraic characterization of the set of instances of (StQP) with a positive gap. For a given $Q \in \cS^n$, checking if a decomposition given by \eqref{inexact_decomp} exists is equivalent to solving the corresponding (StQP) instance, which is clearly an intractable problem. On the other hand, by relying on this characterization, we propose a procedure to generate an instance of (StQP) with a positive gap.

By Proposition~\ref{pos_gap_char}, the main ingredient is a matrix $M \in \partial ~{\cal COP}^n \backslash {\cal SPN}^n$. Recall that ${\cal COP}^n = {\cal SPN}^n$ for each $n \leq 4$ by \eqref{diananda}. Therefore, $n = 5$ is the smallest dimension for which ${\cal COP}^n \backslash {\cal SPN}^n \neq \emptyset$. To that end, recall the well-known Horn matrix (see, e.g.,~\cite{hall1963copositive}) given by
\begin{equation} \label{horn}
H = \begin{bmatrix} 1 & -1 & 1 & 1 & -1 \\-1 & 1 & -1 & 1 & 1 \\1 & -1 & 1 & -1 & 1 \\1 & 1 & -1 & 	1 & -1 \\ -1 & 1 & 1 & -1 & 1 \end{bmatrix} \in \partial ~{\cal COP}^5 \backslash {\cal SPN}^5,
\end{equation}
and
\begin{equation} \label{horn_zeros}
\left\{\frac{1}{2}\left(e_i + e_j\right) \in \R^5: (i,j) \in \left\{(1,2),(2,3),(3,4),(4,5),(5,1)\right\}\right\} \subseteq \mathbf{V}^H, 
\end{equation}
where $\mathbf{V}^H$ is given by \eqref{def_M_zeros}.
Note that $H \in \cS^5 \backslash \cQ^5$ by Proposition~\ref{pos_gap_char}. Indeed, we have $\ell(H) \approx -0.1056$ whereas $\nu(H) = 0$.

For any $n \geq 5$, let $B \in \cS^{n - 5}$ and let $C \in \R^{(n-5) \times 5}$ be two matrices such that $B  \in {\cal COP}^{n-5}$ and each entry of $C$ is nonnegative. Note, in particular, that one may choose $B \in {\cal SPN}^{n-5}$ (or even $B \in \cN^{n-5}$) and $C = 0$. Let us define 
\begin{equation} \label{def_M_hat}
\widehat{M} = \begin{bmatrix} B & C \\ C^T & H \end{bmatrix} \in \cS^n.
\end{equation}
By~\cite[Lemma 3.4(a)]{shaked2016spn}, it follows that $\widehat{M} \in {\cal COP}^n$. Finally, let $J \in \R^{n \times n}$ be an arbitrary permutation matrix and let $D \in \cS^n$ be an arbitrary diagonal matrix with strictly positive entries. Let us define
\begin{equation} \label{def_M_matrix}
M =  J D \widehat{M} D J^T,    
\end{equation}
where $\widehat{M}$ is given by \eqref{def_M_hat}. By Lemma~\ref{lem_gen_rels} (ii) and (iii), $M \in {\cal COP}^n$ since $\widehat{M} \in {\cal COP}^n$. Furthermore, $M \not \in {\cal SPN}^n$ since, otherwise, this would imply that $H \in {\cal SPN}^5$ by Lemma~\ref{lem_gen_rels} (ii), (iii), and (iv). Finally, we claim that $M \in \partial ~{\cal COP}^n$. To see this, let $u \in \mathbf{V}^H$, and let us define $\widehat{u} = [0^T,u^T]^T \in \Delta_n$, which implies that $\widehat{u}^T \widehat{M} \widehat{u} = 0$, i.e., $\widehat{M} \in \partial ~{\cal COP}^n$. Therefore, we obtain $v^T M v = 0$, where $v = J D^{-1} \widehat{u} \in \R^n_+ \backslash \{0\}$. Therefore, defining $w = (1/(e^T v))v \in \Delta_n$, we have $w \in \mathbf{V}^M$, where $\mathbf{V}^M$ is defined as in \eqref{def_M_zeros}. It follows that $M \in \partial ~{\cal COP}^n \backslash {\cal SPN}^n$. 

Finally, by picking an arbitrary real number $\lambda \in \R$ and defining $Q = \lambda E + M$, we ensure that $Q \in \cS^n \backslash \cQ^n$, $\nu(Q) = \lambda$, and $\Omega(Q) = \mathbf{V}^M$ by Proposition~\ref{pos_gap_char}. 

We close this section by making two observations. First, suppose that $n \geq 6$. By choosing $B = 0$ in \eqref{def_M_hat} and $P = D = I$ in \eqref{def_M_matrix}, we can guarantee that any $x \in \Delta_n$ of the form $x = [\hat{x}^T,0^T]^T$, where $\hat{x} \in \Delta_{n-5}$ satisfies $x \in \mathbf{V}^M$. Then, by Proposition~\ref{pos_gap_char}, any such $x \in \Delta_n$ would be an optimal solution of the (StQP) instance corresponding to $Q = \lambda E + M$ for any $\lambda \in \R$. Therefore, for $n \geq 6$ and for any $x \in \Delta_n$ such that $|A(x)| \leq n - 5$, where $A(x)$ is given by \eqref{def_P}, one can construct a matrix $Q \in \cS_x \backslash \cQ_x$, where $\cS_x$ and $\cQ_x$ are given by \eqref{def_Sx} and \eqref{def_Qx}, respectively. Second, we 
note that the Horn matrix $H$ in the definition \eqref{def_M_hat} can be replaced by any extreme ray of ${\cal COP}^5$ that does not belong to ${\cal SPN}^5$, which were fully characterized in \cite[Theorem 3.1]{hildebrand2012extreme}.

\section{Concluding Remarks} \label{conc}

In this paper, we studied the doubly nonnegative relaxations of standard quadratic programs. We presented characterizations of instances of (StQP) that admit an exact relaxation as well as those with a positive gap. Both of our characterizations can be used as algorithmic procedures to construct an instance of (StQP) with a prespecified optimal solution, for which the doubly nonnegative relaxation is either exact or has a positive gap. In addition, we explicitly identified three families of instances with exact relaxations. We also established several properties between the maximal cliques of the convexity graph and the tightness of the relaxation.

For a given $Q \in \cS^n$, consider an exact optimal solution of (DN-D), which satisfies $Q = \ell(Q) E + P + N$, where $P \in {\cal PSD}^n$ and ${\cal N} \in \cN^n$. We can check if there exists $x \in \Delta_n$ such that $Q \in \cQ_x$, where $\cQ_x$ is given by \eqref{def_Qx}, by solving the following feasibility problem:
\[
Px = 0, \quad x^T N x = 0, \quad x \in \Delta_n.
\]
This problem can easily be cast as the following mixed integer linear feasibility problem:
\[
\begin{array}{lcl}
Px & = & 0, \\
e^T x & = & 1, \\
x_j & \leq & y_j, \quad j = 1,\ldots,n,\\
x_j & = & 0, \quad j \in \{1,\ldots,n\} \textrm{ s.t. } N_{jj} > 0,\\
y_i + y_j & \leq & 1, \quad 1 \leq i < j \leq n \textrm{ s.t. } N_{ij} > 0,\\
x & \geq & 0,\\
y_j & \in & \{0,1\}, \quad j = 1,\ldots,n.
\end{array}
\]
However, this procedure requires an exact solution of the dual problem (DN-D) and does not shed light on the existence of a polynomial-time algorithm for the membership problem in $\cQ^n$. 

Another interesting research direction is the investigation of the topological properties of $\cQ^n$ as well as the set $\cS^n \backslash \cQ^n$. We intend to study these problems in the near future.

\bibliographystyle{abbrv}
\bibliography{references}
\end{document}